\documentclass{amsart}

\usepackage{amsmath}
\usepackage{amssymb}
\usepackage{amsthm}
\usepackage{url}
\usepackage{MnSymbol}

\usepackage{hyperref}

\usepackage{tikz}
\usetikzlibrary{arrows}

\newtheorem{theorem}{Theorem}[section]
\newtheorem{lemma}[theorem]{Lemma}

\newtheorem{claim}[theorem]{Claim}

\theoremstyle{definition}

\newtheorem{question}[theorem]{Question}

\theoremstyle{remark}

\newcommand{\mc}[1]{\mathcal{#1}}

\DeclareMathOperator{\direction}{direction}
\DeclareMathOperator{\target}{target}
\DeclareMathOperator{\LEFT}{left}
\DeclareMathOperator{\RIGHT}{right}
\DeclareMathOperator{\extent}{scope}

\begin{document}

\title{There is no classification of the decidably presentable structures}

\author{Matthew Harrison-Trainor}
\address{Group in Logic and the Methodology of Science\\
University of California, Berkeley\\
 USA}
\email{matthew.h-t@berkeley.edu}
\urladdr{\href{http://www.math.berkeley.edu/~mattht/index.html}{www.math.berkeley.edu/$\sim$mattht}}

\begin{abstract}
A computable structure $\mc{A}$ is decidable if, given a formula $\varphi(\bar{x})$ of elementary first-order logic, and a tuple $\bar{a} \in \mc{A}$, we have a decision procedure to decide whether $\varphi$ holds of $\bar{a}$. We show that there is no reasonable classification of the decidably presentable structures. Formally, we show that the index set of the computable structures with decidable presentations is $\Sigma^1_1$-complete. This result holds even if we restrict out attention to groups, graphs, or fields. We also show that the index sets of the computable structures with $n$-decidable presentations is $\Sigma^1_1$-complete for any $n$.
\end{abstract}

\maketitle

\section{Introduction}

In effective mathematics, we are concerned with \textit{computable structures}. A mathematical structure---a set together with operations and relations on that set---is \textit{computable} if the set and the operations and relations on it are all computable. For example, a computable field is one where the domain is a computable set and the operations of addition and multiplication are computable. In a computable structure, we can effectively answer quantifier-free questions, such as, for elements $a$, $b$, and $c$ of a field, whether $a + b = b \cdot c$.

There are many other questions about a structure that we might want to answer in a computable way. For example, in a field, we might want to be able to decide whether a given polynomial has a root. In general, this is undecidable, but sometimes, such as for algebraically closed fields, this can be done. In fact, given a computable algebraically closed field, as a result of quantifier elimination we can decide the answer to any question that can be formulated in elementary first-order logic, i.e., as a logical formula using $\vee$, $\wedge$, $\neg$, $\rightarrow$, $\forall$, and $\exists$. In general, we say that a computable structure is \textit{decidable} if there is a method to compute, given elements $a_1,\ldots,a_n$ and a formula $\varphi$ of elementary first-order logic, whether $\varphi$ holds of $a_1,\ldots,a_n$. Every computable algebraically closed field is decidable.

An important phenomenon in computability theory is that there can be computable structures which are isomorphic, but not computably isomorphic, so that we cannot transfer computational properties from on to the other. For example, the standard presentation of the linear order $(\mathbb{N},<)$ is decidable. However, there is also a computable copy of the same structure in which the successor relation is not computable, and hence this copy is not decidable. (Here, $a$ is a succesor of $b$ if and only if $(\forall c)[ c < b \vee c > a]$.) Though these two computable structures are isomorphic, they are not computably isomorphic.

This paper is about the problem of characterizing those computable structures which have a decidable presentation. This problem was probably first stated by Goncharov, and has more recently been posed for example by Bazhenov at the 2015 Mal'cev Meeting and Fokina at the 2016 ASL meeting in Storrs, CT. We will show that there is no such characterization.

More formally, our main theorem is as follows. Fix an effective list of the diagrams of the (partial) computable structures.
\begin{theorem}\label{thm:d-pres}
The index set
\[ I_{\text{d-pres}} = \{i \mid \text{the $i$th computable structure is decidably presentable}\}\]
is $\Sigma^1_1$-complete.
\end{theorem}
\noindent This theorem is proved in Section \ref{sec:d}.

As a result, there is no possible reasonable characterization of the computable structures with decidable presentations. What we mean is that there is no simpler way to check whether a computable structure $\mc{A}$ has a decidable presentation than to ask: \textit{Does there exists a decidable structure $\mc{B}$ and a classical isomorphism between $\mc{A}$ and $\mc{B}$}? This requires searching through all possible isomorphisms, of which there may be continuum-many, between $\mc{A}$ and $\mc{B}$. (Contrast this with a very naive, and incorrect, candidate for a characterization: A computable structure $\mc{A}$ has a decidable copy if and only if there is a computable listing of the types it realizes. In this case, we must look through the countably many possible computable listings of types, and check whether they list the types in $\mc{A}$. This requires only quantifiers over natural numbers, and objects which can be coded by natural numbers.\footnote{For some restricted classes of structures, such a characterization might be possible. For example, Andrews \cite{Andrews} showed that if $\mc{M}$ is a model of a decidable $\omega$-stable theory with countably many countable models, then $\mc{M}$ has a decidable copy if and only if all of the types realized in $\mc{M}$ are recursive.}) If there were a simpler characterization of the computable structures with decidable presentations, then one would expect that characterization to yield a simpler way of checking whether a computable structure has a decidable presentation.

A similar approach was taken in \cite{DKLLMT}, where it was shown that there is no reasonable characterization of computable categoricity, and in \cite{DowneyMontalban08}, where it was shown that there is no reasonable classification of abelian groups. This approach originated with \cite{GoncharovKnight}. See also \cite{LemppSlaman,Fokina07,CFGKKMP,FGKKT,GBM,GoncharovBazhenovMarchuk15}.

\subsection{Decidable presentability in familiar classes}

What if we are interested in a specific class of structures, such as fields, graphs, or groups? In the case of some particular classes of structures---which are \textit{universal} in a sense soon to be described---it follows immediately from  Theorem \ref{thm:d-pres} that there is no classification of the decidably presentable structures in that class.

Hirschfeldt, Khoussainov, Shore, and Slinko \cite{HKSS} showed that many classes of structures---such as graphs and groups---are \textit{universal}. What we mean when we say that a class $\mc{C}$ of structures is universal is that for every structure $\mc{A}$, there is a structure $\mc{B} \in \mc{C}$ such that $\mc{A}$ and $\mc{B}$ are \textit{effectively bi-interpretable}. That is, each is interpretable in the other using computable infinitary $\Sigma_1$ formulas, in a compatible way, and so for most computability-theoretic purposes, the two structures are the same. Equivalently---see \cite{HTM3}---there is a computable bi-transformation between copies of $\mc{A}$ and copies of $\mc{B}$, i.e., there is a computable way to turn presentations of $\mc{A}$ into presentations of $\mc{B}$, and vice versa, in a functorial way. We note that this exact definition of a universal class did not appear in \cite{HKSS}, but comes from later work in \cite[Section 3]{MPPSS} and \cite[Definition 5.4]{MonICM}. (For groups, we need to add finitely many constants to the language.)

It turns out, when we examine the proofs from \cite{HKSS}, that the effective bi-interpretations between a structure $\mc{A}$ and the corresponding graph (or group) $\mc{G}_{\mc{A}}$, use only elementary first-order formulas. (An effective bi-interpretation is, in general, allowed to use infinitary formulas.) That means that a structure $\mc{A}$ is decidable if and only if the corresponding graph (or group) $\mc{G}_{\mc{A}}$ is decidable.

Miller, Park, Poonen, Schoutens, and Shlapentokh \cite{MPPSS} recently showed that the class of fields is also universal. Again, the bi-interpretations between a structure and the corresponding field use only elementary first-order formulas.

It follows that one cannot characterize which graphs, groups, and fields are decidably presentable.

\begin{theorem}
The index sets of the decidably presentable graphs, groups, and fields are $\Sigma^1_1$-complete.
\end{theorem}

Other familiar classes of structures, such as linear orders and boolean algebras, are not universal, and so a similar argument does not work. It is possible that such classes admit a characterization of the decidably presentable structures in that class. For linear orders in particular, this question has already been raised:

\begin{question}[Moses, see \cite{Cholak00}]\label{ques:moses}
Can one characterize the linear orderings which have a decidable copy?
\end{question}

\noindent We believe that the Friedman-Stanley \cite{FriedmanStnaley89} transformation $T$ of structures into linear orders preserves decidability, in the sense that $\mc{A}$ is decidable if and only if $T(\mc{A})$ is decidable. It would follow that the answer to this question is ``no''.

Abelian groups are another class of structures which is not universal. However, it is still an open question whether or not abelian groups are Borel complete.

\begin{question}\label{ques:abel}
Can one characterize the torsion-free abelian groups which have a decidable copy?
\end{question}

\subsection{\texorpdfstring{$n$-decidable structures}{n-decidable structures}}

One can also ask whether a computable structure has an $n$-decidable copy. An \textit{$n$-decidable structure} is a structure in which we can decide whether a formula $\varphi$, with $n$ alternations of quantifiers, holds of a tuple $a_1,\ldots,a_\ell$. For each $n$, there are $n$-decidable structures which are have no $n+1$-decidable copies, and there is a structure which has $n$-decidable copies for all $n$, but no decidable copy \cite{ChisholmMoses}. We say that a structure is \textit{$n$-presentable} is it has an $n$-decidable copy. There is no simpler characterization of the $n$-presentable structures.

\begin{theorem}\label{thm:n-pres}
For each $n \in \omega$, the index set
\[ I_{\text{$n$-pres}} = \{i \mid \text{the $i$th computable structure is $n$-presentable}\}\]
is $\Sigma^1_1$-complete.
\end{theorem}

\noindent The proof of Theorem \ref{thm:n-pres} will be simpler than the proof of Theorem \ref{thm:d-pres}, and so we will begin by proving Theorem \ref{thm:n-pres} in Section \ref{sec:1}. To prove Theorem \ref{thm:d-pres}, we must also guess at $\Sigma^0_2$ facts.

\subsection{Further questions}

In addition to Questions \ref{ques:moses} and \ref{ques:abel} above, there are many more questions to resolve. In \cite{DKLLMT}, it was shown that the index set of the computably categorical structures is $\Pi^1_1$ complete. In \cite{GoncharovBazhenovMarchuk15}, this was extended to show that the index set of the computable structures with computable dimension $n$ is $\Pi^1_1$-complete, for finite $n$. The case of $n = \omega$ is still open. 
\begin{question}
What is the complexity of the index set of the computable structures with computable dimension $\omega$?
\end{question}

In this paper, we considered structures which have one computable copy which is decidable. One could also consider structures all of whose computable copies are decidable. We call such a structure \textit{intrinsically decidable}. One can, as usual, also define a notion of relative intrinsic decidability: A structure is \textit{relatively intrinsically decidable} if, for every isomorphic copy $\mc{A}$ of that structure, the elementary diagram of $\mc{A}$ is computable in $\mathop{deg}(\mc{A})$. By the uniform version of a theorem of Ash, Knight, Manasse, and Slaman \cite{AKMS}, and independently Chisholm \cite{Chisholm}, a computable structure $\mc{A}$ is relatively intrinsically decidable if and only if it has a sort of quantifier elimination: Every elementary first-order definable subset of $\mc{A}$ is (uniformly) definable by a computable infinitary $\Sigma_1$ formula, and also by a computable infinitary $\Pi_1$ formula.
One expects there to be structures which are intrinsically decidable but not relatively intrinsically decidable, as there are, for example, structures which are computably categorical but not relatively computably categorical \cite{Goncharov77}.
Note that deciding whether a structure is relatively intrinsically decidable is arithmetic; however, one might guess that intrinsic decidability is actually $\Sigma^1_1$ complete.
\begin{question}
What is the complexity of the index set of the computable structures all of whose computable copies are decidable?
\end{question}

See also Question \ref{ques:red} which we state later after providing sufficient context.

\section{Some useful lemmas}

\subsection{A sequence of structures}

It is well-known that there are computable structures $\mc{C}_{\infty}$ such that the index set of the computable structures which are isomorphic to $\mc{C}_{\infty}$ is $\Sigma^1_1$-complete. A small modification of the same argument, which we will repeat below in brief, shows that the same is true of decidable structures: There is a decidable structure $\mc{C}_{\infty}$ such that the index set of the decidable structures which are isomorphic to $\mc{C}_{\infty}$ is $\Sigma^1_1$-complete. We will use these structures in the constructions for Theorems \ref{thm:d-pres} and \ref{thm:n-pres}.

To build the structure $\mc{C}_{\infty}$ we will use the following lemma, which is probably folklore; similar results appear in, for example, \cite{Ash91}.

\begin{lemma}
Given a computable linear order $\mathcal{L}$, we can, uniformly in $\mathcal{L}$, build a decidable copy of $\omega^\omega \cdot (1 + \mathcal{L})$.
\end{lemma}
\begin{proof}
It is well-known that there is a decidable copy, which we will call $\mathcal{W}$, of $\omega^\omega$; we may also choose $\mathcal{W}$ so that $\mathcal{W} + \mathcal{W}$ is decidable. Define $\mathcal{A} = \mathcal{W} \cdot (1 + \mathcal{L})$. We represent elements of $\mathcal{A}$ as pairs $(l,w)$ with $l \in 1 + \mathcal{L}$ and $w \in \mathcal{W}$, ordered lexicographically starting with $l$. We claim that $\mathcal{A}$ is decidable.

Indeed, given a tuple $\bar{a}$, break up $\bar{a}$ into tuples $\bar{a}_1,\ldots,\bar{a}_n$ where each element of $\bar{a}_i$ is of the form $(l_i,w)$ for some $w \in \mathcal{W}$, and $l_1 < \cdots < l_n$. Let $\bar{a}_i$ consist of the elements $a_1^1 < \cdots < a_1^{m_i}$, and let $w_i^j$ be such that $a_i^j = (l_i,w_i^j)$. Then (see Corollary 13.39 of \cite{Rosenstein82}) the complete type of $\bar{a}$ is determined effectively by the elementary first-order theories of the intervals
\[ (-\infty,a_1^1],[a_1^1,a_1^2],\ldots,[a_1^{m_1},a_2^1],[a_2^1,a_2^2],\ldots,[a_2^{m_2},a_3^1],\ldots,[a_n^{m_n},\infty).\]
Each interval $[a_i^j,a_i^{j+1}]$ has the same order type as $[w_i^j,w_i^{j+1}]$ which is decidable, as it is a definable subset of $\mathcal{W}$. The order type of $[a_i^{m_i},a_{i+1}^1]$ is $\omega^\omega \cdot [l_{i},l_{i+1}) + w_{i+1}^1$, which has the same theory as $\omega^\omega + w_{i+1}^1$ (see Theorem 6.21 of \cite{Rosenstein82}); this theory is decidable. The interval $(-\infty,a_1^1]$ has the same theory as either $w_1^1$ (if $l_1$ is smaller than $\mc{L}$) or $\omega^\omega + w_{1}^1$ (if $l_1 \in \mc{L}$). Finally, the interval $[a_n^{m_n},\infty)$ has the same theory as $\omega^\omega$. Thus the type of $\bar{a}$ is computable in $\mc{A}$, and so $\mc{A}$ is decidable. 
\end{proof}

\begin{lemma}\label{lem:seq}
Let $S$ be a $\Sigma^1_1$ set. There is a decidable structure $\mc{C}_\infty$ and a uniformly decidable sequence of structures $(\mc{C}_n)_{n \in \omega}$ such that $\mc{C}_n \cong \mc{C}_\infty$ if and only if $n \in S$. All of these structures are in the same language.
\end{lemma}
\begin{proof}
Harrison \cite{Harrison68} constructed a computable linear order $\mathcal{H}$ of order type $\omega_1^{CK} (1 + \mathbb{Q})$. From \cite[Lemma 5.2]{CholakDowneyHarrington08} or \cite[Theorem 4.4(d)]{GoncharovKnight02}, we get a computable sequence of computable linear orders $(\mathcal{L}_n)_{n \in \omega}$ such that $\mathcal{L}_n$ is isomorphic to $\mathcal{H}$ if and only if $n \in S$. Then letting $\mathcal{C}_n$ be a decidable copy of $\omega^\omega \cdot (1 + \mathcal{L}_n)$, we get a uniformly decidable sequence of structures $(\mathcal{C}_n)_{n \in \omega}$. (We take $\mathcal{C}_\infty$ to be a decidable copy of $\mathcal{H}$, which is isomorphic to $\omega^\omega \cdot (1 + \mathcal{H})$.) If $\mathcal{L}_n$ was well-founded, so is $\mathcal{C}_n$, and if $\mathcal{L}_n$ was isomorphic to $\mathcal{H}$, then so is $\mathcal{C}_n$.
\end{proof}

\subsection{Building decidable structures from disjoint unions}

In this section, we will prove three lemmas about constructing a decidable structure by taking disjoint unions of other decidable structures. We will use these lemmas during the construction.

\begin{lemma}\label{lem:disj}
Let $\mc{A}_1,\ldots,\mc{A}_k$ be decidable structures. Then the disjoint union $\mc{B}$ of $\mc{A}_1,\ldots,\mc{A}_k$, with relations $R_1,\ldots,R_k$ picking out the domains of $\mc{A}_1,\ldots,\mc{A}_k$ respectively, is also decidable. This is uniform.
\end{lemma}
\begin{proof}
It suffices to show that $\mc{B}$ is decidable with respect to the many-sorted logic with sorts defined by $R_1,\ldots,R_k$. The many-sorted logic has quantifiers which range only over a single sort $R_i$, and the relations of a structure $\mc{A}_i$ are restricted to the sort $R_i$. Indeed, it is easy to translate any formula in the single-sorted language of $\mc{B}$ to an equivalent formula in the many-sorted language. In what follows, by an $\mc{A}_i$-formula we mean a formula involving only the sort $\mc{A}_i$.

We can easily argue by induction on formulas that each formula $\varphi$ in the many-sorted language of $\mc{B}$ is equivalent to a boolean combination of $\mc{A}_i$-formulas. For example, if $\varphi \equiv (\exists x \in R_p) \psi$, and (placing the boolean combination equivalent to $\psi$ in disjunctive normal form)
\[ \psi \equiv \bigvee_{i=1}^r \bigwedge_{j = 1}^k \theta_{i,j} \]
where $\theta_{i,j}$ is a $\mc{A}_j$-formula, we get that
\[ \varphi \equiv \bigvee_{i=1}^r \bigwedge_{j = 1}^k \theta_{i,j}' \]
where $\theta_{i,p}' = (\exists x \in R_p)\theta_{i,p}$ and $\theta_{i,j}' = \theta_{i,j}$ if $j \neq p$.

Then given a formula $\varphi$ in the many-sorted language of $\mc{B}$, write $\varphi$ as a boolean combination of $\mc{A}_i$-formulas:
\[ \varphi \equiv \bigvee_{i=1}^r \bigwedge_{j = 1}^k \theta_{i,j} \]
where $\theta_{i,j}$ is a $\mc{A}_j$-formula.
We can decide the truth of each $\theta_{i,j}$ as each $\mc{A}_j$ is decidable, and hence we can decide the truth of $\varphi$.
\end{proof}

A slightly more complicated argument proves the following similar lemma.

\begin{lemma}\label{lem:disj-mixed}
Let $\mc{A}_1,\ldots,\mc{A}_k$ be decidable structures. Then the disjoint union $\mc{B}$ of $\mc{A}_1,\ldots,\mc{A}_k$, with an equivalence relation $E$ whose equivalence classes pick out the structures $\mc{A}_i$, is also decidable. This is uniform.
\end{lemma}
\begin{proof}[Proof sketch]
The structure $\mc{B}$ is effectively bi-interpretable, using first-order formulas, with the structure from the previous lemma after naming one element from each of the $k$ equivalence classes.
\end{proof}

The third, and final, lemma allows us to take the disjoint union of infinitely many structures, as long as they are all elementarily equivalent.

\begin{lemma}\label{lem:disj-inf}
Let $(\mc{A}_i)_{i \in \omega}$ be a sequence of uniformly decidable structures.
Suppose that for each $i$ and $j$, $\mc{A}_i \equiv \mc{A}_j$.
Let $\mc{B}$ be the disjoint union of the $\mc{A}_i$, with an equivalence relation $E$ whose equivalence classes pick out the structures $\mc{A}_i$.
Then $\mc{B}$ is decidable. This is uniform.
\end{lemma}
\begin{proof}
View the structures as relational structures. Given a formula $\varphi(x_1,\ldots,x_\ell)$ and $a_1,\ldots,a_\ell$, we need to decide whether $\mc{B} \models \varphi(a_1,\ldots,a_\ell)$. Let $n$ be the quantifier depth of $\varphi$. Let $\mc{B}^*$ be substructure of $\mc{B}$ which consists of those structures $\mc{A}_i$ containing $a_1,\ldots,a_\ell$ and $n$ other structures $\mc{A}_i$. We claim that $\mc{B} \models \varphi(a_1,\ldots,a_\ell)$ if and only if $\mc{B}^* \models \varphi(a_1,\ldots,a_\ell)$. Since $\mc{B}^*$ is decidable, uniformly in $n$, by the previous lemma, we can decide whether $\mc{B} \models \varphi(a_1,\ldots,a_\ell)$. Thus $\mc{B}$ is decidable, and this is uniform.

To see that $\mc{B} \models \varphi(a_1,\ldots,a_\ell)$ if and only if $\mc{B}^* \models \varphi(a_1,\ldots,a_\ell)$, we can play the Ehrenfeucht-Fra\"ss\'e game with depth $n$. Denote by $\mc{M} \overset{r}{\sim} \mc{N}$ that Duplicator has a winning strategy for the Ehrenfreucht-Fra\"sse\'e game with $r$ moves, i.e., that $\mc{M}$ and $\mc{N}$ satisfy the same formulas with quantifier depth $r$. We want to show that $(\mc{B}^*;a_1,\ldots,a_\ell) \overset{n}{\sim} (\mc{B};a_1,\ldots,a_\ell)$. To prove this, it is more convenient to prove a stronger claim

\begin{claim}
Given $r$ and $m$ with $r + m \leq n + \ell$, tuples $\bar{x}_1 \in A_{j_1},\ldots,\bar{x}_m \in A_{j_m}$, all in $\mc{B}^*$, and $\bar{y}_1 \in A_{k_1},\ldots,\bar{y}_m \in A_{k_m}$ (with no repetition among the lists of the structures), $(\mc{B}^*;\bar{x}_1,\ldots,\bar{x}_m) \overset{r}{\sim} (\mc{B};\bar{y}_1,\ldots,\bar{y}_m)$ if and only if for each $i$, $(\mc{A}_{j_i};\bar{x}_i) \overset{r}{\sim} (\mc{A}_{k_i};\bar{y}_i)$.
\end{claim}

From this, if we take $r = n$ and (rearranging $a_1,\ldots,a_\ell$) take \[ (a_1,\ldots,a_\ell) = (\bar{x}_1,\ldots,\bar{x}_m) = (\bar{y}_1,\ldots,\bar{y}_m) \] with $j_i = k_i$ for all $i$, then we immediately get that $(\mc{B}^*;a_1,\ldots,a_\ell) \overset{n}{\sim} (\mc{B};a_1,\ldots,a_\ell)$ as desired. So the proof of the claim will finish the proof of the lemma.

\medskip{}

\textit{Proof of claim.} The proof of this claim is by induction on $r$. For $r = 0$, $\bar{x}_1,\ldots,\bar{x}_m$ satisfy the same atomic formulas in $\mc{B}^*$ as $\bar{y}_1,\ldots,\bar{y}_m$ do in $\mc{B}$ if and only if for each $i$, $\bar{x}_i \in \mc{A}_{j_i}$ satisfies the same atomic formulas in $\mc{A}_{j_i}$ as $\bar{y}_i$ does in $\mc{A}_{k_i}$. Given $r > 0$, it is clear that if $(\mc{B}^*;\bar{x}_1,\ldots,\bar{x}_m) \overset{r}{\sim} (\mc{B};\bar{y}_1,\ldots,\bar{y}_m)$ then for each $i$, $(\mc{A}_{j_i};\bar{x}_i) \overset{r}{\sim} (\mc{A}_{k_i};\bar{y}_i)$. For the other direction, suppose that for each $i$, $(\mc{A}_{j_i};\bar{x}_i) \overset{r}{\sim} (\mc{A}_{k_i};\bar{y}_i)$. Given $y' \in \mc{B}$, we must find $x' \in \mc{B}^*$ such that $(\mc{B}^*;\bar{x}_1,\ldots,\bar{x}_m,x') \overset{r-1}{\sim} (\mc{B};\bar{y}_1,\ldots,\bar{y}_m,y')$. (The other case---finding $y' \in \mc{B}$ given $x' \in \mc{B}^*$---is similar and actually easier.) 

\textbf{Case 1}. If $y' \in A_{k_i}$ for some $i = 1,\ldots,m$, then since $(\mc{A}_{j_i};\bar{x}_i) \overset{r}{\sim} (\mc{A}_{k_i};\bar{y}_i)$, there is $x' \in A_{j_i}$ such that $(\mc{A}_{j_i};\bar{x}_i x') \overset{r-1}{\sim} (\mc{A}_{k_i};\bar{y}_i y')$. Thus, by the induction hypothesis, $(\mc{B}^*;\bar{x}_1,\ldots,\bar{x}_m,x') \overset{r-1}{\sim} (\mc{B};\bar{y}_1,\ldots,\bar{y}_m,y')$.

\textbf{Case 2}. Otherwise, let $k_{m+1}$ be such that $y' \in A_{k_{m+1}}$. Since $r + m \leq n + \ell$, we can choose $j_{m+1}$ different from $j_1,\ldots,j_m$ such that $A_{j_{m+1}}$ is included in $\mc{B}^*$. Since $A_{k_{m+1}} \equiv A_{j_{m+1}}$, we can find $x' \in A_{j_{m+1}}$ such that $(A_{k_{m+1}},y') \overset{r-1}{\sim} (A_{j_{m+1}},x')$. We then have, with $\bar{x}_{m+1} = x'$ and $\bar{y}_{m+1} = y'$, that $(\mc{A}_{j_i};\bar{x}_i) \overset{r-1}{\sim} (\mc{A}_{k_i};\bar{y}_i)$ for $i = 1,\ldots,m+1$ and that $(r-1) + (m+1) \leq n + \ell$. So $(\mc{B}^*;\bar{x}_1,\ldots,\bar{x}_m,x') \overset{r-1}{\sim} (\mc{B};\bar{y}_1,\ldots,\bar{y}_m,y')$ by the induction hypothesis.
\end{proof}

\section{1-presentable structures}\label{sec:1}

In this section, we will prove the case $n=1$ of Theorem \ref{thm:n-pres}: The index set of 1-presentable structures is $\Sigma^1_1$-complete. The general case is essentially the same, but restricting to the case $n=1$ will make the proof more readable, and, in fact, the case $n \geq 2$ will follow from the proof of Theorem \ref{thm:d-pres}. (See Section \ref{sec:sec-part}.)

Fix a $\Sigma^1_1$ set $S$. We must build a uniformly computable sequence of computable structures $(\mc{M}_n)_{n \in \omega}$ such that $\mc{M}_n$ is 1-presentable if and only if $n \in S$. Fix, as in Lemma \ref{lem:seq}, decidable structures $\mc{C}_n$ and $\mc{C}_\infty$ such that $\mc{C}_n \cong \mc{C}_{\infty}$ if and only if $n \in S$. We will use $\mc{C}_n$ and $\mc{C}_{\infty}$ in the construction of $\mc{M}_n$. Also fix a computable listing $(\mc{D}_i)_{i \in \omega}$ of the (possibly partial) 1-diagrams of the 1-decidable structures.

The structures $\mc{M}_n$ will be the disjoint union of infinitely many structures $(\mc{A}_i)_{i \in \omega}$, each distinguished in $\mc{M}_n$ by some unary relation $P_i$. (We may assume that each of the structures $\mc{D}_i$ is a partial structure of this form.) There are two properties that we want from the construction of $\mc{A}_i$:
\begin{enumerate}
	\item If $n \in S$, then $\mc{A}_i$ will have a $1$-decidable presentation uniformly in $i$.
	\item If $n \notin S$ and $\mc{D}_i$ is a 1-decidable structure, then $\mc{A}_i$ will not be isomorphic to the structure with domain $P_i$ in the 1-decidable structure $\mc{D}_i$.
\end{enumerate}
Thus, if $n \in S$, then we can build a 1-decidable presentation of $\mc{M}_n$ by building 1-decidable copies of each $\mc{A}_i$. On the other hand, if $n \notin S$, then $\mc{M}_n$ is not 1-presentable as it cannot be isomorphic to any 1-decidable structure $\mc{D}_i$.

For the remainder of the construction, we can fix $i$. For simplicity, denote $\mc{A}_i$ by $\mc{A}$ and let $\mc{D}$ be the structure with domain $P_i$ in $\mc{D}_i$. So we want to build $\mc{A}$ so that if $n \in S$, then $\mc{A}$ will have a $1$-decidable presentation (which we can construct uniformly), and if $n \notin S$, then $\mc{A}$ is not isomorphic to $\mc{D}$.

\subsection{\texorpdfstring{$\Sigma^0_1$ labeling of 1-decidable structures}{c.e. labeling of 1-decidable structures}}
\label{sec:labels}

Given a 1-decidable structure $\mc{A}$, we will describe how to add labels to $\mc{A}$ which are $\Sigma^0_1$ over the 1-diagram of $\mc{A}$ using a construction which is essentially a Marker extension \cite{Marker}. Intuitively, what we want to do is as follows. We want to be able to attach labels to elements of $\mc{A}$ in a c.e.\ way---that is, so that at any stage, we can add a label to a node---so that the resulting structure, with the labels attached, is also 1-decidable, and so that in the 1-diagram of an isomorphic copy of $\mc{A}$, we can enumerate the labels.

More formally, fix an infinite computable set $\mc{L}$ of labels. Given a sequence of subsets $X = (X_{\ell})_{\ell \in \mc{L}}$ of $\mc{A}$, we want to define a three-sorted structure $\mc{A}^X$, whose first sort is just the structure $\mc{A}$, as follows. We will refer to the sorts as $\mc{A}$, $\mc{S}_1$, and $\mc{S}_2$. The language of $\mc{A}^X$ will be the language of $\mc{A}$ augmented with functions $f \colon \mc{S}_1 \to \mc{A}$ and $g \colon \mc{S}_2 \to \mc{S}_1$, a unary relation $U^\ell \subseteq \mc{S}_1$ for each $\ell \in \omega$, and a unary relation $R \subseteq \mc{S}_2$.

For each element $x$ of $\mc{A}$, there will be infinitely many elements $y$ of the second sort $\mc{S}_1$ with $f(y) = x$. These will be partitioned into infinitely many disjoint sets $U^\ell$ for $\ell \in \omega$. Each element of $\mc{S}_1$ will be the pre-image, under $f$, of some $x \in \mc{A}$.

For each element $y$ of $\mc{S}_1$, there will be infinitely many elements $z \in \mc{S}_2$ with $g(z) = y$, and each element of $\mc{S}_2$ will be the pre-image, under $g$, of some $y \in \mc{S}_1$.

For every $x \in \mc{A}$, there will be infinitely many $y \in f^{-1}(x) \cap U^\ell$ such that there are infinitely many $z \in g^{-1}(y)$ with $R(z)$, and infinitely many $z \in g^{-1}(y)$ with $\neg R(z)$.
If $x \notin X_\ell$, this will be the case for all $y \in f^{-1}(x) \cap U^\ell$, but if $x \in X_\ell$, then there will also be infinitely many $y \in f^{-1}(x) \cap U^\ell$ such that for all $z \in g^{-1}(y)$, $R(z)$.

The next three lemmas show that this construction does what we want it to do.

\begin{lemma}\label{lem:ea-def}
Let $\mc{A}$ be a structure and let $X = (X_{\ell})_{\ell \in \mc{L}}$ be subsets of $\mc{A}$. The sets $X_{\ell}$ are definable in $\mc{A}^X$ by $\exists \forall$ formulas, and these formulas are uniform in $\ell$ and independent of $\mc{A}$ or $X$.
\end{lemma}
\begin{proof}
The set $X_\ell$ is definable as the subset of the first sort of $\mc{A}^X$ defined by $(\exists y \in \mc{S}_1) \left[ f(y) = x \wedge U^\ell(y) \wedge (\forall z \in \mc{S}_2)(g(z) = y \rightarrow R(z))\right]$.
\end{proof}

\begin{lemma}\label{lem:lab-comp}
Let $\mc{A}$ be a computable structure and let $X = (X_{\ell})_{\ell \in \mc{L}}$ be a computable sequence of codes for c.e.\ subsets of $\mc{A}$. Then, uniformly in $X$ and in the atomic diagram of $\mc{A}$, we can build a computable copy of $\mc{A}^X$.
\end{lemma}
\begin{proof}
The copy of $\mc{A}^X$ we build will have the computable copy of $\mc{A}$ in the first sort, the second sort will contain elements $(x,\ell,s,t)$, and the third sort will contain elements $(x,\ell,s,t,u)$. We define
\begin{align*}
U^\ell &= \{ (x,\ell,s,t) \in S_1 \} \\
f &\colon \mc{S}_2 \to \mc{S}_1 \text{ defined by } (x,\ell,s,t,u) \mapsto (x,\ell,s,t) \\
g &\colon \mc{S}_1 \to \mc{A} \text{ defined by } (x,\ell,s,t) \mapsto x.
\end{align*}
It only remains to define the relation $R$. Given $s$, $t$, and $u$, we will have $R(x,\ell,s,t,u)$ if and only if $u$ is even or if $u$ is odd and $x$ enters $X_\ell$ exactly at stage $s$.
\end{proof}

\begin{lemma}\label{lem:lab-1dec}
Let $\mc{A}$ be a 1-decidable structure and let $X = (X_{\ell})_{\ell \in \mc{L}}$ be a computable sequence of codes for c.e.\ subsets of $\mc{A}$. Then, uniformly in $X$ and in the 1-diagram of $\mc{A}$, we can build a 1-decidable copy of $\mc{A}^X$.
\end{lemma}
\begin{proof}
We can build a 1-decidable copy of $\mc{A}^X$ by putting the 1-decidable copy of $\mc{A}$ in the first sort, and defining the second and third sorts as in the previous lemma. Given a tuple $\bar{a} \in \mc{A}^X$ and an existential formula $(\exists \bar{y}) \varphi(\bar{x},\bar{y})$, we want to decide whether $\mc{A}^X \models (\exists \bar{y}) \varphi(\bar{a},\bar{y})$. First, we may rewrite $\varphi$ in the language where we replace the language of $\mc{A}$ with the predicates
\[ P^{\theta(x_1,\ldots,x_n)} = \{ (a_1,\ldots,a_n) \in \mc{A}^n : \mc{A} \models \theta(a_1,\ldots,a_n) \} \]
where $\theta$ is an existential formula in the language of $\mc{A}$. Next, we may assume that $\varphi$ is a conjunction of atomic formulas.

We will show that $(\exists \bar{y}) \varphi(\bar{x},\bar{y})$ is equivalent, in $\mc{A}^X$, to a quantifier-free formula $\psi(\bar{x})$ in an expanded language with the predicate
\[Q = \{ (x,\ell,s,t) \in \mc{S}_1 : x \notin X_{\ell,\text{at} s} \} \]
which is only allowed to appear positively.
Note that the predicates $Q$ and $P^\theta$ are computable in $\mc{A}^X$, and so we can decide whether $\mc{A}^X \models \psi(\bar{a})$, and hence whether $\mc{A}^X \models (\exists \bar{y}) \varphi(\bar{a},\bar{y})$.

Arguing by induction, it suffices to show that if $\bar{a}$ is a tuple from $\mc{A}^X$, and $\varphi(x_1,\ldots,x_n)$ is a quantifier-free formula in which $Q$ appears only positively, then $(\exists x_n) \varphi(x_1,\ldots,x_n)$ is equivalent in $\mc{A}^X$ to a formula $\psi(x_1,\ldots,x_{n-1})$ in which $Q$ appears only positively.

Since every element of $\mc{A}$ is the image of an element of $\mc{S}_1$ under $g$, and every element of $\mc{S}_1$ is the image of an element of $\mc{S}_2$ under $f$, we may assume that $x_1,\ldots,x_n$ are from the sort $\mc{S}_2$. We may write $\varphi(x_1,\ldots,x_n)$ in the following form:
\begin{align*}
 P^{\theta(y_1,\ldots,y_n)}(f(g(x_1)),\ldots,f(g(x_n))) \wedge \left[ \bigwedge_{i \in I(Q)} Q(g(x_i)) \right] \\
 \wedge \left[ \bigwedge_{i \in I(U_\ell)} U^\ell(g(x_i)) \right] \wedge \left[ \bigwedge_{i \in I(\neg U^\ell)} \neg U^\ell(g(x_i)) \right] \wedge \left[ \bigwedge_{i \in I(R)} R(x_i) \right] \wedge \left[ \bigwedge_{i \in I(\neg R)} \neg R(x_i) \right] \\
\wedge \left[ \bigwedge_{\{i,j\} \in J_1^{=}} x_i = x_j \right] \wedge \left[ \bigwedge_{\{i,j\} \in J_1^{\neq}} x_i \neq x_j \right] \\ \wedge \left[ \bigwedge_{\{i,j\} \in J_2^{=}} g(x_i) = g(x_j) \right] \wedge \left[ \bigwedge_{\{i,j\} \in J_2^{\neq}} g(x_i) \neq g(x_j) \right] \\
\wedge \left[ \bigwedge_{\{i,j\} \in J_3^{=}} f(g(x_i)) = f(g(x_j)) \right] \wedge \left[ \bigwedge_{\{i,j\} \in J_3^{\neq}} f(g(x_i)) \neq f(g(x_j)) \right].
\end{align*}
So that we can refer to it later, let $\chi(x_1,\ldots,x_n)$ be the part of this formula after $P^\theta(f(g(x_1)),\ldots,f(g(x_n)))$. We may assume that $\varphi$ is looks consistent in the sense that $I(U^\ell)$ and $I(\neg U^\ell)$ are disjoint, $I(R)$ and $I(\neg R)$ are disjoint, and so on.

\smallskip{}

\noindent \textbf{Case 1.} If $\{n,i\} \in J_1^=$ for some $i$, then $(\exists x_n) \varphi(x_1,\ldots,x_n)$ is clearly equivalent to $\varphi(x_1,\ldots,x_{n-1},x_i)$.

\smallskip{}

\noindent \textbf{Case 2.} Otherwise, if $\{n,i\} \in J_2^=$ for some $i$, then $(\exists x_n) \varphi(x_1,\ldots,x_n)$ is equivalent to
\begin{align*}
 P^\theta(f(g(x_1)),\ldots,f(g(x_{n-1})),f(g(x_i))) \wedge Q(g(x_i)) \wedge \chi'(x_1,\ldots,x_{n-1})
\end{align*}
if $n \in I(\neg R)$, and 
\begin{align*}
 P^{\theta(y_1,\ldots,y_n)}(f(g(x_1)),\ldots,f(g(x_{n-1})),f(g(x_i)))  \wedge \chi'(x_1,\ldots,x_{n-1})
\end{align*}
otherwise, where $\chi'(x_1,\ldots,x_{n-1})$ is $\chi(x_1,\ldots,x_n)$ with $g(x_n)$ replaced by $g(x_i)$ everywhere, and any term involving only $x_n$ (but not $g(x_n)$, or $f(g(x_n))$) deleted.

\smallskip{}

\noindent \textbf{Case 3.} Otherwise, if $\{n,i\} \in J_3^=$ for some $i$, then $(\exists x_n) \varphi(x_1,\ldots,x_n)$ is equivalent to
\begin{align*}
 P^{\theta(y_1,\ldots,y_n)}(f(g(x_1)),\ldots,f(g(x_{n-1})),f(g(x_i))) \wedge \chi'(x_1,\ldots,x_{n-1})
\end{align*}
where $\chi'(x_1,\ldots,x_{n-1})$ is $\chi(x_1,\ldots,x_n)$ with $f(g(x_n))$ replaced by $f(g(x_i))$ everywhere, and any term involving only $x_n$ or $g(x_n)$ (but not $f(g(x_n))$) deleted.
\smallskip{}

\noindent \textbf{Case 4.} Otherwise, $(\exists x_n) \varphi(x_1,\ldots,x_n)$ is equivalent to
\begin{align*}
 P^{(\exists y_n) \theta(y_1,\ldots,y_n)}(f(g(x_1)),\ldots,f(g(x_{n-1}))) \wedge \chi'(x_1,\ldots,x_{n-1})
\end{align*}
where $\chi'(x_1,\ldots,x_{n-1})$ is $\chi(x_1,\ldots,x_n)$ with any term involving $x_n$, $g(x_n)$, or $f(g(x_n))$ deleted.
\end{proof}

\subsection{Overview of the construction}

Recall that given a structure $\mc{D}$, we want to build $\mc{A}$ so that if $n \in S$, then $\mc{A}$ will have a $1$-decidable presentation (which we can construct uniformly), and if $n \notin S$, then $\mc{A}$ is not isomorphic to $\mc{D}$.

The structure $\mc{A}$ will actually be of the form $\mc{B}^X$ for some sequence of subsets $X = (X_\ell)_{\ell \in \mc{L}}$ of $\mc{B}$. We will build the diagram of $\mc{B}$ in a computable way while also enumerating the sets $X_\ell$. (Though rather than saying that we put an element $x$ into $X_\ell$, we will say that we put the label $\ell$ on $x$.) By Lemma \ref{lem:lab-comp}, $\mc{A} = \mc{B}^X$ will be a computable structure. At the end, to see that if $n \in S$ then $\mc{A}$ has a 1-decidable presentation, we will use Lemma \ref{lem:lab-1dec}. If $\mc{D}$ is going to be isomorphic to $\mc{A} = \mc{B}^X$, then it will have to be of the form $\mc{E}^Y$ for a sequence of subsets $Y = (Y_\ell)_{\ell \in \mc{L}}$ of $\mc{E}$; by Lemma \ref{lem:ea-def}, using the 1-diagram of $\mc{D}$, we can compute $\mc{E}$ and enumerate the sets $Y_\ell$. So we will diagonalize against the 1-diagram of $\mc{E}$ together with an enumeration of the sequence $Y$. To keep the construction as intuitive as possible, we will not mention $\mc{B}$ and $\mc{E}$. Instead, we will think of $\mc{A}$ and $\mc{D}$ as computable structures with c.e.\ labels.

We will now describe the language and general form of $\mc{A}$. There will be a set $N$ of \textit{nodes}. To each node $\nu$, we attach two other structures: a structure in the language of Lemma \ref{lem:seq} with domain $T_\nu$ and a linear order with domain $W_\nu$. $T_\nu$ will be isomorphic to either $\mc{C}_n$ or $\mc{C}_\infty$, and $W_\nu$ will be isomorphic to one of $\omega$, $\omega^*$, or $\omega^* + \omega$. We call $T_\nu$ the \textit{tag} of $\nu$, and we say that the elements of $T_\nu$ are the \textit{$T$-elements} of $\nu$. To each node $\nu$, we associate the structure consisting of $T_\nu$ and $W_\nu$. We call this structure the \textit{$\nu$-component} of $\mc{A}$.

Note that if $n \notin S$, and one node $\nu$ is tagged with $\mc{C}_n$, and a second node $\nu'$ is tagged with $\mc{C}_{\infty}$, then there is no automorphism of $\mc{A}$ taking $\nu$ to $\nu'$, as $\mc{C}_n$ and $\mc{C}_{\infty}$ are not isomorphic. On the other hand, if $n \in S$, then the $\nu$-component and the $\nu'$-component might be isomorphic. The linear orders $W_\nu$ will be used to diagonalize against 1-presentations; in a 1-presentation, a maximal (or minimal) element of a linear order will be distinguished by a universal formula, while in a computable presentation we can always change our mind between building a copy of $\omega$ or $\omega^*$.

To the nodes $\nu$, and to the $T$-elements, we attach \textit{labels} which are $\Sigma^0_1$ over the 1-diagram in the sense described in Section \ref{sec:labels}. We have infinitely many labels $\ell_k$ and a distinguished label $L$. These labels will be used in the same way that labels are used to build computably categorical structures (\cite{DKLLMT}) or structures of finite computable dimension (\cite{Goncharov80b}), and we suggest that it might help the reader who is not familiar with this technique read one of these papers before proceeding. At each stage $s$, each node $\nu$ which is of the form $\rho$ or $\sigma_i^\Box$ (these will be defined later), and each of their $T$-elements, will have two labels $\ell_k$ which are unique to them; one label will be the \textit{primary label} and the other the \textit{secondary label}. There will be other labels in the \textit{bag} which hold of every element. The bag will begin empty. The nodes $\tau_{i,s}^{\Box}$, and their $T$-elements, will all be labeled in the same was as $\rho$ was at stage $s$ (except that they may also be labeled with $L$). The nodes $\rho$ and $\sigma_i^\Box$, and their $T$-elements, will never be labeled $L$.

While it looks like $\mc{D}$ is copying $\mc{A}$, we will periodically add the primary labels of each element to the bag, labeling every element with them, and then give each element a new unique label. What were the secondary labels will become the new primary labels, and the new labels will be the new secondary labels. If infinitely often we add the primary labels to the bag then at the end of the construction every element will be labeled with the same labels---those in the bag. But at every finite stage of the construction, every element will be distinguished.

\subsection{\texorpdfstring{The construction of $\mc{A}$}{The construction of A}}

We begin at stage $s = 0$. To start, put into $\mc{A}$ the distinguished node $\rho$, and the other nodes $(\sigma_i^{\leftmapsto})$, $(\sigma_i^{\mapsto})$, and $(\sigma_i^{\leftrightarrow})$, and $(\tau_{i,0}^{\leftmapsto})$ and $(\tau_{i,0}^{\mapsto})$. At later stages of the construction, we will add new nodes $(\tau_{i,s}^{\leftmapsto})$ and $(\tau_{i,s}^{\mapsto})$ for other values of $s$.

For the node $\rho$: Let $T_\rho$ contain a copy of $\mc{C}_n$, and let $W_\rho$ begin with a single element.
For each node $\sigma_i^{\Box}$: Let $T_{\sigma_i^{\Box}}$ contain a copy of $\mc{C}_{\infty}$, and let $W_{\sigma_i^{\Box}}$ contain a linear order which depends on $\Box$: for $\Box = \leftmapsto$, set $W_{\sigma_i^{\leftmapsto}} = \omega^*$; for $\Box = \mapsto$, set $W_{\sigma_i^{\mapsto}} = \omega$; and for $\Box = \leftrightarrow$, set $W_{\sigma_i^{\leftrightarrow}} = \omega^* + \omega$.
The nodes $\tau_{i,0}^{\Box}$ will be the same as the nodes $\sigma_i^{\Box}$, except that $T_{\tau_{i,0}^{\Box}}$ will contain a copy of $\mc{C}_n$ instead of $\mc{C}_{\infty}$.
For every node $\nu$ other than $\rho$, the $\nu$-component of $\mc{A}$ will be a 1-decidable structure. (Note also that there are no relations that hold between different components.) Indeed, as soon as we add a node $\nu$ (other than $\rho$) to the domain, we will immediately completely decide $T_\nu$ and $W_\nu$. Later, we may add labels to the elements, but since the labels are $\Sigma^0_1$ over the 1-diagram, this is 1-decidable.

\medskip{}

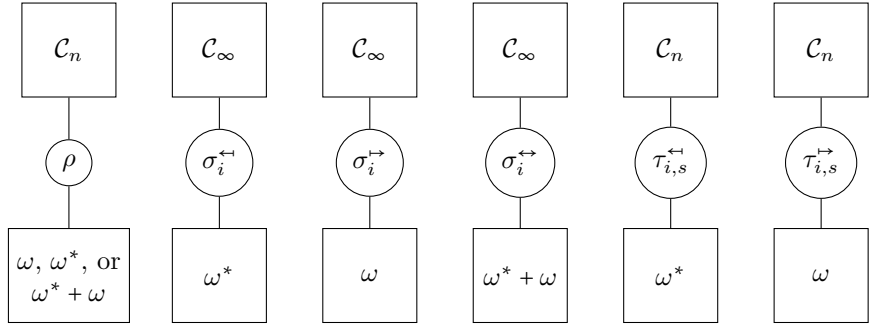
\begin{tikzpicture}[node distance=2.5cm,auto,>=latex']
    \node[circle,draw=black,minimum width=.5cm, minimum height=.5cm] (rho) at (0,0) {$\rho$};
		\node[rectangle,draw=black,minimum width=1.25cm, minimum height=1.25cm] (rho-c) at (0,1.5) {$\mc{C}_n$};
		\node[align=center,rectangle,draw=black,minimum width=1.25cm, minimum height=1.25cm] (rho-w) at (0,-1.5) {$\omega$, $\omega^*$, or\\$\omega^*+\omega$};
    \draw[-] (rho) -- (rho-c);
		\draw[-] (rho) -- (rho-w);
		
		\node[circle,draw=black,minimum width=.5cm, minimum height=.5cm] (sigmal) at (2,0) {$\sigma^{\leftmapsto}_i$};
		\node[rectangle,draw=black,minimum width=1.25cm, minimum height=1.25cm] (sigmal-c) at (2,1.5) {$\mc{C}_\infty$};
		\node[rectangle,draw=black,minimum width=1.25cm, minimum height=1.25cm] (sigmal-w) at (2,-1.5) {$\omega^*$};
    \draw[-] (sigmal) -- (sigmal-c);
		\draw[-] (sigmal) -- (sigmal-w);
		
		\node[circle,draw=black,minimum width=.5cm, minimum height=.5cm] (sigmar) at (4,0) {$\sigma^{\mapsto}_i$};
		\node[rectangle,draw=black,minimum width=1.25cm, minimum height=1.25cm] (sigmar-c) at (4,1.5) {$\mc{C}_\infty$};
		\node[rectangle,draw=black,minimum width=1.25cm, minimum height=1.25cm] (sigmar-w) at (4,-1.5) {$\omega$};
    \draw[-] (sigmar) -- (sigmar-c);
		\draw[-] (sigmar) -- (sigmar-w);
		
		\node[circle,draw=black,minimum width=.5cm, minimum height=.5cm] (sigmarl) at (6,0) {$\sigma^{\leftrightarrow}_i$};
		\node[rectangle,draw=black,minimum width=1.25cm, minimum height=1.25cm] (sigmarl-c) at (6,1.5) {$\mc{C}_\infty$};
		\node[rectangle,draw=black,minimum width=1.25cm, minimum height=1.25cm] (sigmarl-w) at (6,-1.5) {$\omega^*+\omega$};
    \draw[-] (sigmarl) -- (sigmarl-c);
		\draw[-] (sigmarl) -- (sigmarl-w);
		
		\node[circle,draw=black,minimum width=.5cm, minimum height=.5cm] (taul) at (8,0) {$\tau^{\leftmapsto}_{i,s}$};
		\node[rectangle,draw=black,minimum width=1.25cm, minimum height=1.25cm] (taul-c) at (8,1.5) {$\mc{C}_n$};
		\node[rectangle,draw=black,minimum width=1.25cm, minimum height=1.25cm] (taul-w) at (8,-1.5) {$\omega^*$};
    \draw[-] (taul) -- (taul-c);
		\draw[-] (taul) -- (taul-w);
		
		\node[circle,draw=black,minimum width=.5cm, minimum height=.5cm] (taur) at (10,0) {$\tau^{\mapsto}_{i,s}$};
		\node[rectangle,draw=black,minimum width=1.25cm, minimum height=1.25cm] (taur-c) at (10,1.5) {$\mc{C}_n$};
		\node[rectangle,draw=black,minimum width=1.25cm, minimum height=1.25cm] (taur-w) at (10,-1.5) {$\omega$};
    \draw[-] (taur) -- (taur-c);
		\draw[-] (taur) -- (taur-w);
\end{tikzpicture}

\medskip{}

Assign, to each of the nodes $\rho$ and $\sigma_i^{\Box}$, and to each of their $T$-elements, two unique labels $\ell_k$. Label the $\tau_{i,0}^{\Box}$ in the same way as $\rho$. It will always be true at each stage $s$ that every node $\rho$ and $\sigma_i^{\Box}$ and each of their $T$-elements will have two unique labels that distinguish them from every other such element. No nodes will be labeled by $L$ at this point. The bag begins empty.

Certain stages will be \textit{expansionary stages}. The expansionary stages are those where we get more evidence that $\mc{A}$ is isomorphic to $\mc{D}$. The stage $0$ is an expansionary stage by definition. At each expansionary stage $s$, we will have a number $\extent(s)$ which measures how much of the structures $\mc{A}$ and $\mc{D}$ we are looking at. Begin with $\extent(0) = 0$.

At each stage $s$, we will have a \textit{target}, $\target(s)$, for $\rho$. The target is a node of $\mc{D}$ which we think is the image, under isomorphism, of $\rho$. We will try to make $W_\rho$ different from the target. We do this by choosing a \textit{direction}, $\direction(s)$, for $\rho$ at stage $s$, which is either \textit{left} or \textit{right}. If the direction is left, then we are trying to build $W_\rho$ to be a copy of $\omega^*$; if it is right, then we are trying to build a copy of $\omega$. We will update the target and direction only at expansionary stages. At every stage, expansionary or not, we will add a single element to $W_\rho$ depending on the direction at that stage. Thus $W_\rho$ will end up being isomorphic to $\omega$, $\omega^*$, or $\omega^* + \omega$.

\medskip{}

The general idea of the construction is as follows, when $\mc{D}$ is a total 1-decidable structure, in each of the two cases $n \in S$ and $n \notin S$. If $n \notin S$, then $\mc{C}_n$ and $\mc{C}_{\infty}$ are not isomorphic. So the node $\rho$ is fixed by every automorphism of $\mc{A}$. If we can identify the image of $\rho$ in $\mc{D}$, and have it be our target for all sufficiently large stages, then we will diagonalize against $\mc{D}$ by making $W_\rho$ different from the target in $\mc{D}$. Of course, the only thing distinguishing $\rho$ from the $\sigma_i^{\Box}$ is that it is tagged with $\mc{C}_n$ instead of $\mc{C}_{\infty}$, and these two structures may look very similar. This is where we use the labels: In $\mc{A}$, we give $\rho$ a label that distinguishes it from all of the other nodes, and so $\mc{D}$ must produce a node which looks similar; we use this node as the target. Then, if $\mc{D}$ copies the labels we put on $\mc{A}$, we can force it to also tag the target node with $\mc{C}_n$, making our diagonalization successful. Of course, in the limit, everything ends up with the same labels; and the $r_{i,s}^{\Box}$ are labeled $L$, so that they can be distinguished from $\rho$.

If $n \in S$, then $\mc{C}_n$ and $\mc{C}_{\infty}$ are isomorphic. First, if there are infinitely many expansionary stages, then all of the nodes and $T$-elements end up tagged the same. If $W_\rho$ is isomorphic to $\omega$, then the $\rho$-component is isomorphic to each $\sigma_i^{\mapsto}$-component; so we could have built a copy of $\mc{A}$ without ever having built the $\rho$-component! The $\sigma_i^{\mapsto}$-components are actually 1-decidable, since we decide everything about them (except the labels, which are $\Sigma^0_1$ over the 1-diagram) as soon as we add them to the structure. Thus we can build a 1-decidable copy of $\mc{A}$. The same argument works if $W_\rho$ is isomorphic to $\omega^*$ or to $\omega^* + \omega$. Unfortunately, if there are only finitely many expansionary stages, then the nodes and $T$-elements may end up having different labels. But in this case, after the last expansionary stage $s$, we never add any more labels, and so the $\rho$-component will be isomorphic to each $\tau_{i,s}^{\mapsto}$- or $\tau_{i,s}^{\leftmapsto}$-component, and again we could have built a 1-decidable copy of $\mc{A}$ by not building the $\rho$-component.

\medskip{}

\noindent \textit{Construction at stage $s$.} At stage $s$, so far we have built $\mc{A}[s-1]$. The first thing we do at stage $s$ is to decide whether the stage $s$ is expansionary. Let $s^*$ be the last expansionary stage. Stage $s$ is expansionary if there are:
\begin{enumerate}
	\item nodes $\nu_0,\ldots,\nu_r$ of $\mc{A}[s-1]$, containing among them the first $\extent(s^*)$ nodes of $\mc{A}[s-1]$;
	\item $T$-elements $\bar{a}_0 \in T_{\nu_0},\ldots,\bar{a}_r \in T_{\nu_r}$, containing among them the first $\extent(s^*)$ elements of each of these components;
	\item nodes $\mu_0,\ldots,\mu_r$ of $\mc{B}[s]$, containing among them the first $\extent(s^*)$ nodes of $\mc{D}[s]$; and
	\item $T$-elements $\bar{d}_0 \in T_{\mu_0},\ldots,\bar{d}_r \in T_{\mu_r}$, containing among them the first $\extent(s^*)$ elements of each of these components
\end{enumerate}
such that
\begin{itemize}
	\item the atomic types of $\nu_0,\ldots,\nu_r;\bar{a}_0,\ldots,\bar{a}_r$ in $\mc{A}[s-1]$ and $\mu_0,\ldots,\mu_r;\bar{d}_0,\ldots,\bar{d}_r$ in $\mc{D}[s]$ are the same, and
	\item each of the elements from $\nu_0,\ldots,\nu_r;\bar{a}_0,\ldots,\bar{a}_r$ has the same labels in $\mc{A}[s-1]$ as the corresponding elements from $\mu_0,\ldots,\mu_r;\bar{d}_0,\ldots,\bar{d}_r$ have in $\mc{D}[s]$.
\end{itemize}
Otherwise, stage $s$ is not expansionary. If stage $s$ is expansionary, let $\extent(s) \geq \extent(s^*) + 1$ be large enough that $\nu_0,\ldots,\nu_r$ are among the first $\extent(s)$ nodes of $\mc{A}$, $\bar{a}_0,\ldots,\bar{a}_r$ are among the first $\extent(s)$ elements of their components, $\mu_0,\ldots,\mu_r$ are among the first $\extent(s)$ nodes of $\mc{D}$, and $\bar{d}_0,\ldots,\bar{d}_r$ are among the first $\extent(s)$ elements of their components.

If stage $s$ is expansionary, then continue by \textit{updating the target} followed by \textit{renewing labels} as described below. If the stage $s$ is not expansionary, the target and direction are the same as they were at the last expansionary stage. At all stages, expansionary or not, we finish by adding a new element to $W_\rho$. If $\direction(s) = \RIGHT$, add the new element to the right of all existing elements. Otherwise, if $\direction(s) = \LEFT$, add the new element to the left of the existing ones. In this way we obtain the structure $\mc{A}[s]$.

\medskip{}

\noindent \textit{Updating the target.} In $\mc{D}[s]$, find the least node, if one exists, which is labeled exactly by the labels of $\rho$ (and so not by $L$). Set $\target(s)$ to be this node. (If no such element exists, $\target(s)$ is undefined and $\direction(s) = \RIGHT$.)

Now, look at the linear order $W_{\target(s)}$. If it has a greatest element (i.e., an element which the 1-diagram of $\mc{D}[s]$ says is the greatest element), set $\direction(s) = \RIGHT$. Otherwise, set $\direction(s) = \LEFT$.

\medskip{}

\noindent \textit{Renewing labels.} Recall that $s^*$ was the previous expansionary stage. First, apply the label $L$ to each node $\tau_{i,s^*}$. Second, each of the nodes $\rho$ and $\sigma_i^{\Box}$ and their $T$-elements have two labels which only of themselves and which are not in the bag. Add each of the primary labels to the bag. The secondary labels becomes the primary labels. Then, label each of these elements with each label from the bag along with a new unique secondary label.

Build new nodes $\tau_{i,s}^{\leftmapsto}$ and $\tau_{i,s}^{\mapsto}$ tagged with copies of $\mc{C}_n$. Attach a copy of $\omega^*$ or $\omega$ to each of these nodes respectively.
Label these nodes and their $T$-elements in the same way that $\rho$ and its $T$-elements are currently labeled.

\subsection{The verification}

\begin{lemma}
$W_\rho$ is isomorphic to either $\omega$, $\omega^*$, or $\omega^* + \omega$. These three cases correspond, respectively, to having $\direction(s) = \RIGHT$ for all but finitely many $s$, $\direction(s) = \LEFT$ for all but finitely many $s$, and $\direction(s) = \RIGHT$ and $\direction(s) = \LEFT$ for infinitely many $s$ each.
\end{lemma}
\begin{proof}
At each stage $s$ we add a single element to $W_\rho$ on either the left or right hand side, depending on the direction.
\end{proof}

Note that the direction can only change at an expansionary stage, so that if there are only finitely many expansionary stages, $W_\rho$ is isomorphic to either $\omega$ or $\omega^*$. This is why we only add nodes $\tau_{i,s}^{\mapsto}$ and $\tau_{i,s}^{\leftmapsto}$, but not $\tau_{i,s}^{\leftrightarrow}$.

\begin{lemma}\label{lem:inf-exp}
If $\mc{A}$ is isomorphic to $\mc{D}$, then there are infinitely many expansionary stages.
\end{lemma}
\begin{proof}
Suppose to the contrary that there is a last expansionary stage $s^*$, and that $\mc{A}$ is isomorphic to $\mc{D}$ at the end of the construction, say by an isomorphism $f$. Then after stage $s^*$, we never add any more nodes into $\mc{A}$, and we never add any new labels to any elements. Let $\mu_0,\ldots,\mu_r$ be the first $\extent(s^*)$ nodes of $\mc{A}$ together with the inverse images, under $f$, of the first $\extent(s^*)$ nodes of $\mc{D}$. Let $\bar{a}_0 \in T_{\mu_0},\ldots,\bar{a}_r \in T_{\mu_r}$ be the first $\extent(s^*)$ elements of these components, together with the inverse images, under $f$, of the first $\extent(s^*)$ elements of $T_{f(\mu_0)},\ldots,T_{f(\mu_r)}$. Then, for sufficiently large $s$, $\mu_0,\ldots,\mu_r;\bar{a}_0,\ldots,\bar{a}_r$ and $f(\mu_0),\ldots,f(\mu_r);f(\bar{a}_0),\ldots,f(\bar{a}_r)$ have the same labels in $\mc{A}[s-1]$ and $\mc{D}[s]$ respectively. Such a stage $s$ is expansionary.
\end{proof}

\begin{lemma}\label{lem:label}
If there are infinitely many expansionary stages, then every node $\rho$ or $\sigma_i^{\Box}$ and their $T$-elements have exactly the same labels. Each $\tau_{i,s}^{\Box}$ is labeled by $L$.
\end{lemma}
\begin{proof}
This lemma is easily seen from the way the labels are renewed in the construction.
\end{proof}

\begin{lemma}
Let $s$ be an expansionary stage and suppose that $a \in \mc{A}$ and $d \in \mc{D}$ are nodes which are among the first $\extent(s)$ nodes of of $\mc{A}$ and $\mc{D}$ respectively (or $T$ elements which are among the first $\extent(s)$ elements of their components, and are associated to nodes which are among the first $\extent(s)$ nodes), and so that $a$ has the same labels in $\mc{A}[s-1]$ as $d$ does in $\mc{D}[s]$.
Then for any expansionary stage $s^* \geq s$, either $a$ and $d$ have the same labels in $\mc{A}[s^* - 1]$ and $\mc{D}[s]$ respectively, or one of them is labeled $L$.
\end{lemma}
\begin{proof}
It suffices to show that if $s^* \geq s$ is an expansionary stage at which $a$ and $d$ have the same labels in $\mc{A}[s^* - 1]$ and $\mc{D}[s]$ respectively, and $s^{**} > s^*$ is the next expansionary stage, then either $a$ and $d$ have the same labels in $\mc{A}[s^{**}-1]$ and $\mc{D}[s^{**}]$ or one of them is labeled $L$.

Let $\ell_{k_1}$ be the primary label of $a$ in $\mc{A}[s^*-1]$, and let $\ell_{k_2}$ be its secondary label. Then by assumption, $d$ is also labeled by $\ell_{k_1}$ and $\ell_{k_2}$ in $\mc{D}[s^*]$.
During stage $s^*$, $\ell_{k_2}$ becomes the primary label of $a$, and $a$ gets a new secondary label $\ell_{k_3}$.
Now at all stages $t$, $s^* < t < s^{**}$, we do not add any labels to elements of $\mc{A}$.
In $\mc{A}[s^{**}-1]$, the only elements labeled $\ell_{k_2}$ are either labeled the same way as $a$, or labeled $L$.
Since $s^{**}$ is an expansionary stage, and $d$ is among the first $\extent(s) < \extent(s^{**})$ nodes of $\mc{D}$ if it is a node (or the first $\extent(s)$ elements of its component, which is among the first $\extent(s)$ components of $\mc{D}$, if $d$ is a $T$-element), there is an element $a' \in \mc{A}[s^{**}-1]$ which is labeled in the same way as $d$.
As $d$ is labeled $\ell_{k_2}$, $a'$ is labeled $\ell_{k_2}$, and so they must both be labeled in the same way as $a$, or be labeled $L$.
\end{proof}

\begin{lemma}\label{lem:bec-iso}
Suppose that $\mc{A}$ and $\mc{D}$ are in fact isomorphic.
Let $s$ be an expansionary stage, and let $\nu$ and $\mu$ be nodes of $\mc{A}$ and $\mc{D}$ respectively, which are among the first $\extent(s)$ nodes of those structures, and assume that neither are ever labeled $L$.
If, at stage $s$, $\nu$ and $\mu$ are labeled in the same way in $\mc{A}[s-1]$ and $\mc{D}[s]$ respectively, then $T_\nu \subseteq \mc{A}$ and $T_\mu \subseteq \mc{D}$ are isomorphic.
\end{lemma}
\begin{proof}
Let $s_0 = s,s_1,s_2,\ldots$ list the expansionary stages after $s$. By the previous lemma, at each expansionary stage $s_i$, $\nu$ and $\mu$ are labeled in the same way in $\mc{A}[s_i-1]$ and $\mc{D}[s_i]$ respectively.

Since $\mc{A}$ and $\mc{D}$ are isomorphic, by Lemma \ref{lem:inf-exp} there are infinitely many expansionary stages.
Given $i$, define a partial isomorphism $f_i \colon T_\nu \to T_\mu$, as follows.
Put a $T$-element $a$, which is among the first $\extent(s_i)$ elements of $T_\nu$, into the domain of $f_i$ if there is $d$ a $T$-element of $\mu$, which is among the first $\extent(s_i)$ elements of $T_\mu$, such that $a$ and $d$ have the same labels in $\mc{A}[s_i-1]$ and $\mc{D}[s]$ respectively.
In this case, set $f_i(a) = d$.
(Note that there can be at most one such $d$ for a given $a$, as no two elements of the same component of $\mc{A}[s_i-1]$ are labeled in the same way.)

\begin{claim}
If $i < i'$, then $f_i \subseteq f_{i'}$.
\end{claim}

Suppose that $f_i(a) = d$. Then $a$ and $d$ are labeled in the same way in $\mc{A}[s_i-1]$ and $\mc{D}[s_i]$ respectively, and are among the first $\extent(s_i)$ elements of $T_\nu$ and $T_\mu$ respectively. Since $\nu$ and $\mu$ are never labeled $L$, neither are $a$ and $d$ at the expansionary stage $s_{i'}$; we will not label $a$ by $L$, and if $d$ was labeled $L$, then $s_{i'}$ could not be an expansionary stage. So by the previous lemma, at the stage $s_{i'}$, $a$ and $d$ are labeled in the same way. Thus we will define $f_{i'}(a) = d$.

\medskip{}

Let $f = \bigcup_{i \in \omega} f_i$.

\begin{claim}
$f$ is one-to-one.
\end{claim}

If $f$ was not one-to-one, then for some $i$, we would have $f_i(a_1) = f_i(a_2) = d$. So then, in $\mc{A}[s_i - 1]$, $a_1$ and $a_2$ are labeled in the same way; but they are both in the same component, and so this cannot happen.

\begin{claim}
$f$ is total and onto.
\end{claim}

To see that $f$ is total, fix $a \in T_\nu$. For some sufficiently large $i$, $a$ will be among the first $\extent(s_i)$ elements of $T_\nu$.
Then, at the next expansionary stage $s_{i+1}$, there will have to be some $\mu';d'$ corresponding (in the sense that they witness that $s_{i+1}$ is a true stage) to $\nu;a$ and $\nu'$ corresponding to $\mu$. Now since $\nu$ and $\mu$ are labeled in the same way, and $\mu$ and $\nu'$ are labeled in the same way, $\nu$ and $\nu'$ are labeled in the same way in $\mc{A}[s_{i+1}-1]$. From the construction, we see that $T_\nu$ and $T_{\nu'}$ are identically either copies of $\mc{C}_n$ or $\mc{C}_\infty$. (The nodes $\nu$ and $\nu'$ might be, for example, $\rho$ and $\tau_{0,s_i}^{\mapsto}$.) Thus there is $a' \in T_{\nu'}$ which corresponds to $a \in T_\nu$, and since $a$ is among the first $\extent(s_i)$ of $T_\nu$, $a'$ is among the first $\extent(s_i)$ elements of $T_{\nu'}$. Also, $\nu'$ is among the first $\extent(s_i)$ nodes of $\mc{A}$.  Thus there is $d \in T_\mu$ which is labeled in the same way as $a'$, which is labeled in the same way as $a$; hence we would set $f_{i+1}(a) = d$.

To see that $f$ is onto, a similar but not identical argument works. Fix $d \in T_\mu$. For some sufficiently large $i$, $a$ will be among the first $\extent(s_i)$ elements of $T_\nu$. Then, at the next expansionary stage $s_{i+1}$, there will have to be some $\nu';a'$ corresponding to $\mu;d$ and $\mu'$ corresponding to $\nu$. Now since $\nu$ and $\mu$ are labeled in the same way, and $\mu$ and $\nu'$ are labeled in the same way, $\nu$ and $\nu'$ are labeled in the same way in $\mc{A}[s_{i+1}-1]$. From the construction, we see that $T_\nu$ and $T_{\nu'}$ are identically either copies of $\mc{C}_n$ or $\mc{C}_\infty$. Thus there is $a \in T_\nu$ which corresponds to $a' \in T_{\nu'}$. Then $d$ is labeled the same way as $a'$, which is labeled in the same way as $a$; hence we would set $f_{i+1}(a) = d$.

\begin{claim}
$f$ is an isomorphism.
\end{claim}

It suffices to show that each $f_i$ is a partial isomorphism. At stage $s_i$, let $a_0,\ldots,a_r$ be the elements in the domain of $f_i$, and let $d_0 = f_i(a_0),\ldots,d_r = f_i(a_r)$. Since $s_i$ is an expansionary stage, there must be elements $a_0',\ldots,a_r'$ of $\mc{A}[s_i-1]$ which are labeled in the same way, and have the same atomic type as $d_0,\ldots,d_r$ in $\mc{D}[s_i]$. But then $a_0',\ldots,a_r'$ are labeled in the same way, in $\mc{A}[s_i-1]$, as $a_0,\ldots,a_r$. We can see from the construction that $a_0,\ldots,a_r$ and $a_0',\ldots,a_r'$ must then have the same atomic type in $\mc{A}[s_i-1]$. (It is possible that $a_0,\ldots,a_r$ are not equal to $a_0',\ldots,a_r'$, for example if the former are in $T_\rho$ and the latter are in $T_{\tau_{0,s_{i-1}}^\mapsto}$.) Hence $f_i$ is a partial isomorphism.

\medskip{}

This finished the proof of the lemma.
\end{proof}

\begin{lemma}\label{lem:succ-dia}
If $n \notin S$, then $\mc{A}$ is not isomorphic to $\mc{D}$.
\end{lemma}
\begin{proof}
Suppose to the contrary that $\mc{A}$ was isomorphic to $\mc{D}$ via an isomorphism $f$. Then by Lemma \ref{lem:inf-exp} there are infinitely many expansionary stages.

Note that $\rho$ is the only node of $\mc{A}$ which is both not labeled $L$ and which is tagged $\mc{C}_n$. Since $\mc{C}_n$ and $\mc{C}_\infty$ are not isomorphic, no node $\sigma_i^{\Box}$ is tagged $\mc{C}_n$, and since there are infinitely many expansionary stages, each $\tau_{i,s}^{\Box}$ is labeled $L$.

Let $d_0,d_1,d_2,\ldots$ list the elements of $\mc{D}$, and let $d_i = f(\rho)$. Let $t$ be a stage after which each of $d_0,\ldots,d_{i-1}$, if it is the image, under $f$, of a node $\tau_{i,s}^{\Box}$ or one of its $T$-elements, is labeled $L$; thus, if one of these elements ever becomes labeled $L$, it does so by stage $t$. Suppose that $t$ is also large enough that $\rho$ and $d_i$ are among the first $\extent(t)$ nodes of $\mc{A}$ and $\mc{D}$ respectively. We claim that for all expansionary stages $s > t$, $\target(s) = d_i$.

Suppose to the contrary that there is an expansionary stage $s$ at which $\target(s) \neq d_i$. Since $\rho$ is among the first $\extent(t)$ nodes of $\mc{A}$, there is at least one $d_j \in \mc{D}[s]$ among the first $\extent(s)$ nodes of $\mc{D}$ which has the same labels as $\rho$ at stage $s$; since $\target(s) \neq d_i$, there is one such $d_j \neq d_i$.

Then either $d_i$ and $\rho$ are labeled differently at stage $s$, or there is a node $d_j$, $j < i$, among the first $\extent(s)$ nodes of of $\mc{D}$, which is labeled in the same way as $d_i$ at stage $s$ (and hence both are labeled in the same way as $\rho$). 

In the first case---if $d_i$ and $\rho$ are labeled differently at stage $s$---then there is another node $\nu \neq \rho$ of $\mc{A}[s-1]$, which is among the first $\extent(s)$ nodes of $\mc{A}$, which is labeled in the same way as $d_i$ is in $\mc{D}[s]$. Note that $d_i$ is not labeled $L$, as $f(\rho) = d_i$. So by Lemma \ref{lem:bec-iso}, $T_{d_i}$ is isomorphic to $T_\nu$; and, since $\nu \neq \rho$, and $\nu$ is not labeled $L$, $\nu$ is of the form $\sigma_i^\Box$ and so $T_\nu$ is isomorphic to $\mc{C}_\infty$. This is a contradiction, as $d_i = f(\rho)$ and $T_\rho$ is isomorphic to $\mc{C}_n$.

In the second case---if there is a node $d_j$, $j < i$, among the first $\extent(s)$ nodes of $\mc{D}$, which is labeled in the same way as $d_i$ at stage $s$---by Lemma \ref{lem:bec-iso}, $T_{d_j}$ and $T_{d_i}$ are both isomorphic to $T_\rho = \mc{C}_n$ and not labeled $L$. But then $\mc{D}$ cannot be isomorphic to $\mc{A}$, as $\rho$ is the only node $\nu$ of $\mc{A}$ not labeled $L$ and with $T_\nu$ isomorphic to $\mc{C}_n$


So for all expansionary stages $s > t$, $\target(s) = d_i$. If $W_{f(\rho)} = \omega^*$, then at some point the greatest element of $W_{f(\rho)}$ is enumerated into $\mc{D}$, and the 1-diagram says that this is the greatest element. Then, from some sufficiently large expansionary stage on, the direction is always $\RIGHT$. Thus $W_\rho = \omega$. On the other hand, if $W_{f(\rho)} = \omega$ or $\omega^* + \omega$, then there is never a greatest element of $W_{f(\rho)}$, and so the direction is always $\LEFT$. Then $W_\rho = \omega^*$. In all cases, $W_\rho$ is not isomorphic to $W_{f(\rho)}$, a contradiction.
\end{proof}

\begin{lemma}\label{lem-cons-unif}
If $n \in S$, then $\mc{A}$ has a $1$-decidable presentation which we can construct uniformly.
\end{lemma}
\begin{proof}
Since $n \in S$, $\mc{C}_n \cong \mc{C}_\infty$. We claim that if we run the construction without building the node $\rho$ and its component, we get a structure $\mc{A}^-$ which is 1-decidable and isomorphic to $\mc{A}$. To see that $\mc{A}^-$ is isomorphic to $\mc{A}$, there are two cases. First, if there are infinitely many expansionary stages then, by Lemma \ref{lem:inf-exp}, $\rho$ and its $T$-elements, and each node $\sigma_i^{\Box}$ and their $T$-elements, all have the same labels. So $\rho$ and its component is actually isomorphic to each of the $\sigma_i^{\Box}$ and their components for the appropriate choice of $\Box$. Since there are infinitely many such nodes, removing $\rho$ does not change the isomorphism type.

On the other hand, if there are only finitely many expansionary stages, then let $s^*$ be the last expansionary stage. After that stage, we never add any more labels. Then $\rho$ and its component is isomorphic to each of the $\tau_{i,s^*}^{\Box}$ and their components for some $\Box \in \{\leftmapsto,\mapsto\}$.

Now we will argue that $\mc{A}^-$ is 1-decidable. By Lemma \ref{lem:lab-1dec}, it suffices to show that the reduct of $\mc{A}^-$ to the language without the labels is 1-decidable (in fact this reduct is decidable), from which it will follow that $\mc{A}^-$ itself, with the labels, is 1-decidable. The rest of the proof of this lemma will be in this smaller language without the labels.

Whenever we add a new node $\nu$ to $\mc{A}^-$, we immediately decide whether $T_\nu = \mc{C}_n$ or $T_\nu = \mc{C}_\infty$, and whether $W_\nu$ is isomorphic to $\omega$, $\omega^*$, or $\omega^* + \omega$. These structures---$\mc{C}_n$, $\mc{C}_\infty$, $\omega$, $\omega^*$, and $\omega^* + \omega$---all have decidable presentations. So the structure which is the disjoint union of $T_\nu$ and $W_\nu$ is decidable, uniformly in $\nu$, by Lemma \ref{lem:disj}. Since this disjoint union is essentially (i.e., up to effective bi-interpretability using finitary $\Delta_0$ formulas) the $\nu$-component, the $\nu$-component is decidable.

By Lemma \ref{lem:disj-inf}, the following five structures are decidable:
\begin{enumerate}
	\item The disjoint union of the $\sigma_i^{\Box}$-components, for a fixed $\Box \in \{\leftmapsto, \mapsto, \leftrightarrow\}$.
	\item The disjoint union of the $\tau_{i,s}^{\Box}$-components, for a fixed $\Box \in \{ \leftmapsto, \mapsto \}$.
\end{enumerate}
Then by Lemma \ref{lem:disj-mixed}, the disjoint union of these five structures is also decidable. This is effectively bi-interpretable, using finitary $\Delta_0$ formulas, to $\mc{A}^-$, which is thus decidable.
\end{proof}

Lemmas \ref{lem:succ-dia} and \ref{lem-cons-unif} are exactly what we wanted from the construction, and complete the proof of Theorem \ref{thm:n-pres}.

\section{Decidably presentable structures}\label{sec:d}

In this section, we will add a guessing argument to the construction from the previous section to show that the index set of decidably presentable structures is $\Sigma^1_1$-complete (Theorem \ref{thm:d-pres}). The new issue that we have to deal with is that the system of labeling which we used previously no longer works with decidable structures, as we cannot make labels which are $\Sigma^0_1$ over the elementary diagram. Instead of labeling elements with existential facts, we will label them by the existence of a non-principal type, which is a $\Sigma^0_2$ fact over the elementary diagram. Then, when examining the decidable structure $\mc{D}$ against which we are diagonalizing, we must guess at the labels.

The argument will also complete the proof of Theorem \ref{thm:n-pres}. See Section \ref{sec:sec-part}.

\subsection{\texorpdfstring{$\Sigma^0_2$ labeling of decidable structures}{Sigma2 labeling of decidable structures}}

This subsection will be analogous to Section \ref{sec:labels}. Once again, fix an infinite computable set $\mc{L}$ of labels. Given a decidable structure $\mc{A}$ and a sequence $X = (X_{\ell})_{\ell \in \mc{L}}$ of subsets of $\mc{A}$, we want to build a two-sorted structure $\mc{A}^X$, whose first sort is just the structure $\mc{A}$, which codes $X$ in a $\Sigma^0_2$ way over the elementary diagram of $\mc{A}$.

We can build $\mc{A}^X$ as follows. $\mc{A}^X$ will again be two-sorted, with the first sort consisting of $\mc{A}$. We will call the second sort $\mc{S}$. The language of $\mc{A}^X$ will be the language of $\mc{A}$ augmented with a function $f \colon \mc{S} \to \mc{A}$, a unary predicate $U^\ell \subseteq \mc{S}$ for each label $\ell$, and infinitely many unary relations $R_i \subseteq \mc{S}$, $i \in \omega$.

The second sort $\mc{S}$ will be partitioned into the pre-images $f^{-1}(x)$ of the elements $x \in \mc{A}$, and each fibre $f^{-1}(x)$ will be partitioned into infinitely many disjoint sets $U^\ell$. If $i < i'$, and $R_{i'}$ holds of an element, then $R_{i}$ will hold of that element, and for each $x$, $\ell$, $i$ there will be infinitely many elements of $f^{-1}(x) \cap U^\ell$ satisfying $R_j$ for $j < i$ but not $R_{i}$. There is a unique non-principal type $p_\ell$ in $f^{-1}(x) \cap U_\ell$ of an element satisfying $R_i$ for all $i$.

We will define the relations $R_i$ such that, given $x \in \mc{A}$ and $\ell$, if $x \in X_\ell$ then there is a single realization of the non-principal type $p_\ell$ in $f^{-1}(x) \cap U^\ell$, and otherwise there will be no realizations of $p_\ell$ in $f^{-1}(x) \cap U^\ell$.

\begin{lemma}\label{lem:s2-def}
Let $\mc{A}$ be a structure and let $X = (X_{\ell})_{\ell \in \mc{L}}$ be a sequence of $\Sigma^0_2$ subsets of $\mc{A}$. The sets $X_{\ell}$ are definable in $\mc{A}^X$ by computable formulas of the form $\exists x \bigwedge_{i \in I} \psi_i(x,\cdot)$, with the $\psi_i$ quantifier-free. These formulas are computable uniformly in $\ell$, and are independent of $\mc{A}$ or $X$.
\end{lemma}
\begin{proof}
The set $X_\ell$ is definable as the subset of the first sort of $\mc{A}^X$ defined by $(\exists y) \left[ f(y) = x \wedge U^\ell(y) \wedge \bigwedge_i R_i(y) \right]$.
\end{proof}

As a result, if $\mc{A}$ is computable, then the sets $X_\ell$ are uniformly $\Sigma^0_2$.

\begin{lemma}\label{lem:s2-comp}
Let $\mc{A}$ be a computable structure and let $X = (X_{\ell})_{\ell \in \mc{L}}$ be a uniform sequence of indices for $\Sigma^0_2$ subsets of $\mc{A}$. Then, uniformly in $X$ and in the atomic diagram of $\mc{A}$, we can build a computable copy of $\mc{A}^X$.
\end{lemma}
\begin{proof}
Let $X_\ell$ be defined by
\[ x \in X_\ell \Longleftrightarrow (\exists y) \left[ (x,y) \in X^\Pi_\ell \right]\]
where $X^\Pi_\ell$ is $\Pi^0_1$ and, if $x \in X_\ell$, then there is a unique $y$ witnessing this. We can find such a set $X^\Pi_\ell$ uniformly in a $\Sigma^0_2$ index for $X_\ell$.

The copy of $\mc{A}^X$ we build will have the decidable copy of $\mc{A}$ in the first sort, and the second sort will contain elements $(x,\ell,s,t)$ and $(x,\ell,\infty,t)$ with $x$ from the first sort and $\ell$, $s$, and $t$ in $\omega$. We will have $f(x,\ell,s,t) = f(x,\ell,\infty,t) = x$ and $U^\ell(x,m,s,t)$ if and only if $m = \ell$. Given $s$, $t$, and $i$, we will have $R_i(x,\ell,s,t)$ if and only if $s < i$. We will have $R_i(x,\ell,\infty,t)$ if and only if $(x,t) \in X^\Pi_\ell$ at stage $i$. This defines a computable copy of $\mc{A}^X$.
\end{proof}

\begin{lemma}\label{lem:s2-dec}
Let $\mc{A}$ be a decidable structure and let $X = (X_{\ell_i})_{i \in \omega}$ be a uniform sequence of indices for $\Sigma^0_2$ subsets of $\mc{A}$. Then, uniformly in $X$ and in the elementary diagram of $\mc{A}$, we can build a the elementary diagram of a decidable copy of $\mc{A}^X$.
\end{lemma}
\begin{proof}
We can build a decidable copy of $\mc{A}^X$ by putting the decidable copy of $\mc{A}$ in the first sort, and defining the second sort as in the previous lemma. This copy of $\mc{A}^X$ is decidable.

For each $\ell$, let $\mc{A}^X[\ell]$ be the reduct of $\mc{A}^X$ which discards all of the predicates $R_i$ except for $R_0,\ldots,R_\ell$. We claim that $\mc{A}^X[\ell]$ is decidable uniformly in $\ell$. From this it will follow that $\mc{A}^X$ is decidable.

These reducts are quite simple structures: Given $x \in \mc{A}$, there are infinitely many elements $y$ of $f^{-1}(x)$, each of which each have, for each $0 \leq i \leq \ell + 1$, infinitely many elements in $g^{-1}(y)$ with $R_j$ for $j < i$ but not $R_{i}$. Thus any two such elements $y$ are isomorphic. A simple argument, in the style of Lemma \ref{lem:lab-1dec} (or Lemma \ref{lem:mark-dec} to follow) but without having to introduce the predicate $Q$, shows that every formula is equivalent in $\mc{A}^X[\ell]$ to a quantifier-free formula in the language with the additional predicate
\[ P^{\theta(y_1,\ldots,y_n)}(x_1,\ldots,x_n) = \{ (a_1,\ldots,a_n) \in \mc{A}^n \colon \mc{A} \models \theta(a_1,\ldots,a_n) \} \]
where $\theta$ is any formula in the language of $\mc{A}$.
\end{proof}

\subsection{The guesses}\label{sec:guesses}

In this section, fix a (possibly partial) decidable structure $\mc{D}$, and a computable sequence $X = (X_\ell)_{\ell \in \mc{L}}$ of indices of $\Sigma^0_2$ subsets of $\mc{D}$, just as one might obtain from a decidable copy of $\mc{D}^X$ as in Lemma \ref{lem:s2-def}. (Even if $\mc{D}$ is a partial structure, we can still obtain a sequence of $\Sigma^0_2$ sets in this way.) We will describe a way of guessing at membership in the sets $X_\ell$. Write
\[ x \in X_\ell \Longleftrightarrow (\exists n) (\forall m) \left[ (x,n,m) \in X_\ell^{\texttt{c}} \right] \]
for some uniformly computable predicates $X_\ell^{\texttt{c}}$. Fix an enumeration of the tuples $(x,\ell,n)$, where $x \in \mc{D}$, $\ell$ is a label, and $n \in \omega$. Assume that in this enumeration, if $(x,\ell,n)$ comes before $(x,\ell,n')$, then $n < n'$.


At each stage $s$, we will have a guess $G_s$ at which elements look like they are in $X_\ell$, and at what the witnesses are. $G_s$ will be a finite set of tuples $(x,\ell,n)$. For each $(x,\ell,n) \in G_s$, we will have that for all $m < s$, $(x,n,m) \in X_\ell^{\texttt{c}}$; the converse will not necessarily be true. If, for all $m < s$, $(x,n,m) \in X_\ell^{\texttt{c}}$, and $n$ is the least such witness, then we say that \textit{$x$ appears to be labeled $\ell$ at stage $s$ with witness $n$}. Note that if, at some stage, $x$ appears to be labeled $\ell$ with witness $n$, and then at some later stage, $x$ does not appear to be labeled $\ell$ with witness $n$, then $x$ can never again appear to be labeled $\ell$ with witness $n$. It is, however, possible for $x$ to not appear to be labeled $\ell$ with witness $n$, then later to appear to be labeled $\ell$ with witness $n$, and then later to again not appear to be labeled $\ell$ with witness $n$.

Begin with $G_0 = \varnothing$. At stage $s$, we will have defined $G_{s^*}$ for $s^* < s$. We must now define $G_s$. If there is some $(x,\ell,n) \in G_{s-1}$ so that $x$ does not appear to be labeled $\ell$ at stage $s$ with witness $n$, then we have made a mistake. In this case, let $t < s$ be greatest such that for each $(x,\ell,n) \in G_{t}$, $x$ appears to be labeled $\ell$ at stage $s$ with witness $n$, and let $G_s = G_t$. Otherwise, if there are no mistakes to correct, let $(x,\ell,n)$ be least (in our fixed enumeration) such that $(x,\ell,n) \notin G_{s-1}$ but $x$ appears to be labeled $\ell$ at stage $s$ with witness $n$. Let $G_s = G_{s-1} \cup \{(x,\ell,n)\}$. (If no such tuple exists, let $G_s = G_{s-1}$.) Note that there is no other $m \neq n$ with $(x,\ell,m) \in G_{s-1}$.

We will borrow some notation from Ash's $\alpha$-systems \cite{AshA,AshB} to talk about the true path. Write $s \leq_0 t$ if and only if $s \leq t$, and $s \leq_1 t$ if $s \leq t$ and $G_s \subseteq G_t$.

\begin{lemma}\label{lem:leq1}
If $s < t < u$, and $s \leq_1 u$, then $s \leq_1 t$.
\end{lemma}
\begin{proof}
Suppose to the contrary that $s \nleq_1 t$, so that $G_s \nsubseteq G_t$. We may assume that $t$ is the least such. So $G_s \subseteq G_{t-1}$. Since $G_{s} \nsubseteq G_t$, we can see from the definition of $G_t$ that there is $(x,\ell,n) \in G_s$ so that $x$ does not appear to be labeled $\ell$ at stage $t$ with witness $n$. By choice of $t$, at stage $t-1$, $x$ appeared to be labeled $\ell$ with witness $n$. So, at stage $u$, that $x$ cannot appear to be labeled $\ell$ with witness $n$, and so $(x,\ell,n)$ cannot be in $G_u$. So $s \nleq_1 u$.
\end{proof}

We say that a stage $s$ is a \textit{true stage} if, for all $t > s$, $s \leq_1 t$.

\begin{lemma}
There are infinitely many true stages.
\end{lemma}
\begin{proof}
Assume that there is a greatest true stage $s$. There is some least $t$ such that $s+1 \nleq_1 t$. Since $s$ is a true stage, $G_s \subseteq G_{s+1},G_t$. By choice of $t$, $G_{s+1} \nsubseteq G_t$; by the minimality of $t$, $G_{s+2},\ldots,G_{t-1} \nsubseteq G_t$ as well. Then we see from the construction that $G_t = G_s$. Thus $t \leq_1 u$ for all $u > t$, contradicting the choice of $s$.
\end{proof}

We call the sequence $s_0 < s_1 < s_2 < \cdots$ of true stages the \textit{true path} of the construction.

\begin{lemma}
If $s$ is a true stage, and $t \leq_1 s$, then $t$ is also a true stage.
\end{lemma}
\begin{proof}
Suppose that $t \leq_1 s$. Then, by Lemma \ref{lem:leq1}, $t \leq_1 s^*$ for all $s^*$ with $t \leq s^* \leq s$; and since $s \leq_1 s^*$ for all $s^* \geq s$, $t \leq_1 s^*$ for all $s^* \geq t$.
\end{proof}

Define $X^s_\ell = \{x \mid (\exists n)\,(x,\ell,n) \in G_s\}$. Note that if $s \leq_1 t$, then $X^s_\ell \subseteq X^t_\ell$. The next lemma will show that the set $X_\ell$ is the union, along the true stages, of the sets $X^s_\ell$.

\begin{lemma}
$X_\ell = \bigcup_{i \in \omega} X^{s_i}_\ell$.
\end{lemma}
\begin{proof}
Note that if $x \notin X_\ell$, then for all $n$, there is $m$ such that $(x,n,m) \notin X_\ell^{\mathtt{c}}$. Fix $n$, and let $m$ be such that $(x,n,m) \notin X_\ell^{\mathtt{c}}$. Thus, for all stages $s > m$, $(x,\ell,n) \notin G_s$; so, for any true stage $t$, $(x,\ell,n) \notin G_t$. Since this is true for all $n$, $x \notin X^s_\ell$ for any true stage $s$.

On the other hand, suppose that $x \in X_\ell$, but for all true stages $s$, $x \notin X^s_\ell$. Since $x \in X_\ell$, for some $n$, for all $m$ we have $(x,n,m) \in X_\ell^{\mathtt{c}}$. And since $x \notin X^s_\ell$ for all true stages $s$, $(x,\ell,n) \notin G_s$ for all true stages $s$. We may assume that $(x,\ell,n)$ is the least such tuple. For some true stage $s$, for all $(x',\ell',n')$ less than $(x,\ell,n)$ in our chosen enumeration, we will either have that $x'$ does not appear to be labeled $\ell'$ as witnessed by $n'$ at all true stages after $s$ (and so $(x',\ell',n')$ can never be in $G_t$ for any $t \geq s$) or that $x' \in X_{\ell'}$ (with least witness $n'$) and $(x',\ell',n') \in G_s$ (so that $(x',\ell',n') \in G_t$ for all $t > s$). So $x$ appears to be labeled $\ell$ as witnessed by $n$ at all stages after $s$. Then at stage $s+1$, we have $G_{s+1} = G_s \cup \{(x,\ell,n)\}$ and $s+1$ is a true stage. So $x \in X^{s+1}_\ell$, a contradiction.
\end{proof}

We will say that a node or $T$-element $x$ from $\mc{D}[s]$ is \textit{labeled $\ell$ (at stage $s$)} if $x \in X^s_\ell$.

\subsection{\texorpdfstring{$\exists \forall$ Marker extensions}{Marker extensions}}

Given a structure $\mc{A}$ together with a relation $X$ on $\mc{A}$, we will describe how to make a certain kind of Marker extension of $(\mc{A},X)$. We will define a three-sorted structure $M(\mc{A},X)$ whose first sort is a copy of the structure $\mc{A}$. Let $n$ be the arity of $X$. We will refer to the sorts as $\mc{A}$, $\mc{S}_1$, and $\mc{S}_2$. The language of $M(\mc{A},X)$ will be the language of $\mc{A}$ augmented with functions $f \colon \mc{S}_1 \to \mc{A}^n$ and $g \colon \mc{S}_2 \to \mc{S}_1$ and a unary relation $R \subseteq \mc{S}_2$.

For each element $\bar{x} \in \mc{A}^n$, there will be infinitely many elements $y$ of the second sort $\mc{S}_1$ with $f(y) = \bar{x}$. Each element of $\mc{S}_1$ will be the pre-image, under $f$, of some $\bar{x} \in \mc{A}^n$. For each element $y$ of $\mc{S}_1$, there will be infinitely many elements $z \in \mc{S}_2$ with $g(z) = y$, and each element of $\mc{S}_2$ will be the pre-image, under $g$, of some $y \in \mc{S}_1$.

For every $\bar{x} \in \mc{A}^n$, there will be infinitely many $y \in f^{-1}(\bar{x})$ such that there are infinitely many $z \in g^{-1}(y)$ with $R(z)$, and infinitely many $z \in g^{-1}(y)$ with $\neg R(z)$.
If $\bar{x} \notin X$, this will be the case for all $y \in f^{-1}(\bar{x})$, but if $\bar{x} \in X$, then there will also be infinitely many $y \in f^{-1}(\bar{x})$ such that for all $z \in g^{-1}(y)$, $R(z)$.

\begin{lemma}\label{lem:mark-def}
$X$ is definable in $M(\mc{A},X)$ by an $\exists \forall$ formula.
\end{lemma}
\begin{proof}
$X$ is defined by the formula
\[ \bar{x} \in X \Longleftrightarrow (\exists y) \left[ f(y) = \bar{x} \wedge (\forall z) \left[ f(z) = y \rightarrow R(z) \right] \right]. \qedhere\]
\end{proof}

\begin{lemma}\label{lem:mark-comp}
If $\mc{A}$ is computable and $X$ is $\Sigma^0_2$, then we can build a computable copy of $M(\mc{A},X)$ uniformly in $\mc{A}$ and $X$.
\end{lemma}
\begin{proof}
Let $X$ be defined by
\[ \bar{x} \in X \Longleftrightarrow (\exists y)(\forall z) \left[ (\bar{x},y,z) \in X^{\texttt{c}} \right]\]
where $X^{\texttt{c}}$ is computable and, if $\bar{x} \in X$, then there are infinitely many $y$ witnessing this (and, for all $y$, if there is $z$ with $(\bar{x},y,z) \notin X^{\texttt{c}}$, then there are infinitely many such $z$). We can find such a set $X^{\texttt{c}}$ uniformly in a $\Sigma^0_2$ index for $X$.

The copy of $M(\mc{A},X)$ we build will have the decidable copy of $\mc{A}$ in the first sort, the second sort will contain elements $(\bar{x},s)$, and the third sort will contain the elements $(\bar{x},s,t)$. We will have $f(\bar{x},s) = \bar{x}$ and $g(\bar{x},s,t) = (\bar{x},s)$. It only remains to define the relation $R$. Given $s$ and $t$, we will have $R(\bar{x},s,t)$ if and only if $(\bar{x},s,t) \in X^{\texttt{c}}$. This defines a computable copy of $M(\mc{A},X)$.
\end{proof}

\begin{lemma}\label{lem:mark-dec}
If $(\mc{A},X)$ is decidable, then we can build a decidable copy of $M(\mc{A},X)$ uniformly in the elementary diagram of $(\mc{A},X)$.
\end{lemma}
\begin{proof}
The copy of $M(\mc{A},X)$ we build will have the decidable copy of $\mc{A}$ in the first sort, the second sort will contain elements $(\bar{x},s)$, and the third sort will contain elements $(\bar{x},s,t)$. We will have $f(\bar{x},s) = \bar{x}$, and $f(\bar{x},s,t) = (\bar{x},s)$. Define $R(\bar{x},s,t)$ if $t$ is odd, or if $s$ and $t$ are even and $\bar{x} \in X$.

We claim that this is decidable. Given a tuple $\bar{a} \in M(\mc{A},X)$ and a formula $\varphi(\bar{x})$, we want to decide whether $M(\mc{A},X) \models \varphi(\bar{a})$. First, we may rewrite $\varphi$ in the language where we replace the language of $\mc{A}$ with the predicates
\[ P^{\theta(y_1,\ldots,y_n)} = \{ (a_1,\ldots,a_n) \in \mc{A}^n : \mc{A} \models \varphi(a_1,\ldots,a_n) \} \]
where $\theta$ is a formula, possibly involving quantifiers, in the language of $\mc{A}$.

We will show that $\varphi(\bar{x})$ is equivalent, in $M(\mc{A},X)$, to a quantifier-free formula $\psi(\bar{x})$ in an expanded language with the predicate
\[Q = \{ (\bar{x},s) \in \mc{S}_1 : \text{$s$ is odd, $t$ is odd, or $\bar{x} \notin X$} \} \]
and the predicates $P^{\theta(y_1,\ldots,y_n)}$, where $\theta$ is now allowed to contain the predicate $R$.
Note that the predicates $Q$ and $P^\theta$ are computable in $M(\mc{A},X)$, and so we can decide whether $M(\mc{A},X) \models \psi(\bar{a})$, and hence whether $M(\mc{A},X) \models \varphi(\bar{a})$.

Arguing by induction, it suffices to show that if $\varphi(x_1,\ldots,x_n)$ is a quantifier-free formula possibly involving $Q$ and $P^\theta$ (where $\theta$ may involve $R$), $(\exists x_n) \varphi(x_1,\ldots,x_n)$ is equivalent in $M(\mc{A},X)$ to a quantifier-free formula $\psi(x_1,\ldots,x_{n-1})$. The argument is essentially the same as Lemma \ref{lem:lab-1dec}, though $f(g(x_n))$ is now a tuple rather than a single element.
\end{proof}

\subsection{Overview of the construction}

As before, fix a $\Sigma^1_1$ set $S$. Given $\mc{D}$ a 2-decidable structure, we want to build a structure $\mc{A}$ so that, if $n \in S$, we can uniformly build a decidable copy of $\mc{A}$, and if $n \notin S$, then $\mc{A}$ is not isomorphic to $\mc{D}$. (We could have taken $\mc{D}$ to be decidable, but by taking it to be 2-decidable we will simultaneously prove the $n\geq 2$ case of Theorem \ref{thm:d-pres}. See Section \ref{sec:sec-part}.)

The structure $\mc{A}$ we construct will be of the form $[M(\mc{B},\preceq)]^X$, where $\preceq$ is a binary relation and $X = (X_\ell)_{\ell \in \mc{L}}$ is a sequence of subsets of $\mc{B}$ (and hence of $M(\mc{B},\preceq)$). We will build $\mc{B}$ in a computable way, with $\preceq$ and the sets $X_\ell$ defined via $\Sigma^0_2$ approximations. By Lemmas \ref{lem:s2-comp} and \ref{lem:mark-comp}, $\mc{A} = [M(\mc{B},\preceq)]^X$ will be computable. To see that if $n \in S$ then $\mc{A}$ has a decidable presentation, we will use Lemmas \ref{lem:s2-dec} and \ref{lem:mark-dec}. If $\mc{D}$ is a total 2-decidable structure which is isomorphic to $\mc{A}$, then $\mc{D} = [M(\mc{E},\precsim)]^Y$ where $\precsim$ is a binary relation on $\mc{E}$ and $Y = (Y_\ell)_{\ell \in \mc{L}}$ is a sequence of subsets of $\mc{E}$. Since $\mc{D}$ is 2-decidable, by Lemma \ref{lem:mark-def}, $\precsim$ is computable. However, the sets $Y_\ell$ may not be computable; so we will have to use the approximations from the Section \ref{sec:guesses}. Recall from that section that we can find a sequence of computable sets $Y^s_\ell$ such that, if $s_0 < s_1 < s_2 < \cdots$ are the true stages, $Y_\ell = \bigcup_{i \in \omega} Y^{s_i}_\ell$. Recall also that we say that a node or $T$-element $x$ from $\mc{D}[s]$ is \textit{labeled $\ell$ (at stage $s$)} if $x \in Y^s_\ell$. Thus, the labels which hold at any true stage are actual labels of elements of $\mc{D}$.

We will describe how to build the structure $\mc{B}$, together with $\Sigma^0_2$ approximations of $\preceq$ and the sets $X = (X_\ell)_{\ell \in \mc{L}}$; for this latter sequence, we will simply talk about labeling elements of $\mc{B}$ by a label $\ell$, by which we mean that we put that element into the set $X_\ell$. Then we will set $\mc{A} = [M(\mc{B},\preceq)]^X$.

The structure $\mc{B}$ will have nodes $\rho$, $\sigma_i^{\Box}$, and $\tau_{i,s}^{\Box}$, each of which has attached to it a copy $T_\nu$ of $\mc{C}_n$ or $\mc{C}_{\infty}$, and a linear order $W_\nu$ which is isomorphic to $\omega$, $\omega^*$, or $\omega^* + \omega$. The linear orders $W_\nu$ will given by the binary relation $\preceq$ with respect to which we take the Marker extension; thus the linear orders are not themselves in the language of $\mc{B}$, but rather are definable by an $\exists\forall$ formula in $M(\mc{B},\preceq)$.

We again have infinitely many labels $\ell_k$ and a distinguished label $L$, but now these labels will be $\Sigma^0_2$ over the elementary diagram of $\mc{A}$.

\subsection{Acting for a guess}

At each stage $s$, our construction will build a partial structure $\mc{B}[s]$, together with a binary relation $\preceq_s$ and labels $\ell_k^s$ and $L^s$.

If $s < t$, then $\mc{B}[t]$ will extend $\mc{B}[s]$. It will not necessarily be true that if $x$ is labeled $\ell$ at stage $s$, then it will be labeled $\ell$ at stage $t$, or that if $x \preceq_s y$, then $x \preceq_t y$. If, in fact, $s \leq_1 t$, then $\preceq_s$ will extend $\preceq_t$, and anything labeled $\ell$ at stage $s$ will still be labeled $\ell$ at stage $t$.

Note that by Lemma \ref{lem:leq1}, if $s < t < u$, $s,t \leq_1 u$, then $s \leq _1 t$. Thus, this last requirement need only be checked at stage $u$ for the greatest $t < u$ with $t \leq_1 u$.

\medskip{}

\noindent \textit{Stage $0$.} Begin at stage $0$ with $\mc{B}[0]$ as follows. In $\mc{B}[0]$, there will be nodes $\rho$ and $\sigma_i^{\Box}$ for $\Box \in \{\leftmapsto,\mapsto,\leftrightarrow\}$. We have $T_\rho = \mc{C}_n$ and $T_{\sigma_i^\Box} = \mc{C}_{\infty}$. We put a single element in $W_\rho$, and in $W_{\sigma_i^\Box}$ we put a linear order $\preceq$ isomorphic to either $\omega$ or $\omega^*$, depending on whether $\Box$ is $\mapsto$ or $\leftmapsto$.

Unlike before, we will not immediately add infinitely many nodes $\tau_{i,0}^{\Box}$, but rather will ``schedule'' two such nodes (one for each of $\Box = \mapsto$ and $\Box = \leftmapsto$) to be added at each stage. We do, at stage $0$, create an infinite \textit{reserve} of nodes which will, at some later stage, become one of the $\tau_{i,s}^{\Box}$. To each of these nodes $\nu$ in the reserve, we have $T_\nu$ be a copy of $\mc{C}_n$, and $W_\nu$ a linear order isomorphic to $\omega$ for half of the nodes, and $\omega^*$ for the other half.

To each node or $T$-element $x$ associated to a node $\rho$ or $\sigma_i^{\Box}$, we choose two unique labels $\ell_1$ and $\ell_2$, as primary and secondary labels, and label $x$ with them.

Set $\extent(0) = 0$.

\medskip{}

\noindent \textit{Action at stage $s$.} Let $s_1,\ldots,s_n < s$ be the previous stages with $s_i \leq_1 s$. We say that these stages are the \textit{$s$-true stages}, and if they were expansionary stages, then we say that they are \textit{$s$-true expansionary stages}. Let $s^*$ be the last $s$-true expansionary stage.

At stage $s$, so far we have built $\mc{B}[s-1]$, $\preceq_{s-1}$ and certain labels $\ell^{s-1}$ on $\mc{B}[s-1]$. The first thing we need to do is to fix any errors that we may have made since the stage $s_n$. So we begin stage $s$ with the order $\preceq_{s_n}$ and only the labels which held at stage $s_n$; any changes to $\preceq$ or the labels after stage $s_n$ and up to, and including, stage $s-1$ are discarded. Also, return all of the nodes $\tau_{i,s'}^{\Box}$, for $s_n < s' < s$, to the reserve.

Now we need to decide whether the stage $s$ is expansionary. Stage $s$ is expansionary if there are:
\begin{enumerate}
	\item nodes $\nu_0,\ldots,\nu_r$ of $\mc{A}[s_n-1]$, containing among them the first $\extent(s^*)$ nodes of $\mc{A}$;
	\item $T$-elements $\bar{a}_0 \in T_{\nu_0},\ldots,\bar{a}_r \in T_{\nu_r}$, containing among them the first $\extent(s^*)$ elements of each of these components;
	\item nodes $\mu_0,\ldots,\mu_r$ of $\mc{B}[s]$, containing among them the first $\extent(s^*)$ nodes of $\mc{D}$; and
	\item $T$-elements $\bar{d}_0 \in T_{\mu_0},\ldots,\bar{d}_r \in T_{\mu_r}$, containing among them the first $\extent(s^*)$ elements of each of these components
\end{enumerate}
such that:
\begin{itemize}
	\item the atomic types of $\nu_0,\ldots,\nu_r;\bar{a}_0,\ldots,\bar{a}_r$ in $\mc{A}[s_n-1]$ and $\mu_0,\ldots,\mu_r;\bar{d}_0,\ldots,\bar{d}_r$ in $\mc{D}[s_n]$ are the same, and
	\item each of the elements from $\nu_0,\ldots,\nu_r;\bar{a}_0,\ldots,\bar{a}_r$ has the same labels in $\mc{A}[s_n-1]$ as the corresponding elements from $\mu_0,\ldots,\mu_r;\bar{d}_0,\ldots,\bar{d}_r$ have in $\mc{D}[s_n]$.
\end{itemize}
Otherwise, stage $s$ is not expansionary. If stage $s$ is expansionary, let $\extent(s) \geq \extent(s^*) + 1$ be large enough that $\nu_0,\ldots,\nu_r$ are among the first $\extent(s)$ nodes of $\mc{A}$, $\bar{a}_0,\ldots,\bar{a}_r$ are among the first $\extent(s)$ elements of their components, $\mu_0,\ldots,\mu_r$ are among the first $\extent(s)$ nodes of $\mc{D}$, and $\bar{d}_0,\ldots,\bar{d}_r$ are among the first $\extent(s)$ elements of their components.

If stage $s$ is expansionary, then continue by \textit{updating the target} followed by \textit{renewing labels} as described below. If the stage $s$ is not expansionary, the target and direction are the same as they were at the last $s$-true expansionary stage.

At all stages, expansionary or not, we finish by adding a new element to the linear order $\preceq$ in $W_\rho$. In $\mc{B}[s_n]$, finitely many of the elements of $W_\rho$ are bear some relation $\preceq$, and these are linearly ordered. If $\direction(s) = \RIGHT$, pick the least element $x$ of $W_\rho$ which does not bear any such relation, and put this new element to the right of the linear order we have built so far. Otherwise, if $\direction(s) = \LEFT$, do the same but add the new element to the left. This defines $\preceq_s$.

Let $s^*$ be the last $s$-true expansionary stage. Take two nodes, which we call $\tau_{s,s^*}^{\leftmapsto}$ and $\tau_{s,s^*}^{\mapsto}$, from the reserve (with $W_\eta$ isomorphic to $\omega^*$ and $\omega$ respectively). Label these with the same labels as $\rho$.

\medskip{}

\noindent \textit{Updating the target.} In $\mc{D}[s]$, find the least node, if one exists, which is labeled exactly by the labels of $\rho$ (and not by $L$). Set $\target(s)$ to be this node. (If no such element exists, $\target(s)$ is undefined and $\direction(s) = \RIGHT$.)

Now, look at the linear order $W_{\target(s)}$. If it has a greatest element, set $\direction(s) = \RIGHT$. Otherwise, set $\direction(s) = \LEFT$. We can recognize whether an element $x$ of $W_{\target(s)}$ is the greatest element by asking whether for all $y$, $x \npreceq y$ (where $x \npreceq y$ is definable by a $\forall\exists$ formula by Lemma \ref{lem:mark-def}). This is a $\forall\exists$ fact, and so we can ask the 2-diagram of $\mc{D}$.

\medskip{}

\noindent \textit{Renewing labels.} Recall that $s^*$ was the previous expansionary stage. First, apply the label $L$ to each node $\tau_{i,s^*}^\Box$. Second, each of the nodes $\rho$ and $\sigma_i^{\Box}$ and their $T$-elements have two labels which only of themselves and which are not in the bag. Add each of the primary labels to the bag. The secondary labels become the primary labels. Then, label each of these elements with each label from the bag along with a new unique secondary label.

\subsection{Verification}

Let $s_0 < s_1 < s_2 < \cdots$ be the true path of the approximation of the labels of $\mc{D}$, i.e., of $Y = (Y_\ell)_{\ell \in \mc{L}}$. During the construction, we defined a computable structure $\mc{B} = \bigcup_s \mc{B}[s]$. We also defined a $\Sigma^0_2$ relation $\preceq\, = \bigcup_{i \in \omega} \preceq_{s_i}$ along the true stages, and a sequence of $\Sigma^0_2$ subsets $X = (X_\ell)_{\ell \in \mc{L}}$ of $\mc{B}$, where $X_\ell = 
\bigcup_{i \in \omega} X^{s_i}_\ell$ and $X^{s}_\ell$ consists of the elements of $\mc{B}$ which were labeled $\ell$ at stage $s$. To see that $\preceq$ and the $X_\ell$ are in fact $\Sigma^0_2$ sets, note that the set of true stages is a $\Pi^0_1$ subset of $\omega$: $s$ is a true stage if and only if, for all $t > s$, $s \leq_1 t$. Then, for example, $x \preceq y$ if and only if there is a true stage $s$ such that $x \preceq_s y$. Thus we can define $\mc{A} = [M(\mc{B},\preceq)]^X$.

The construction ends up being essentially the same as the 1-decidable case along the true stages. Note that $\preceq$ and $X$ were defined so that:
\begin{enumerate}
	\item An element $x \in \mc{B}$ labeled $\ell$ in $\mc{A}$ if and only if it was labeled $\ell$ at some true stage $s$.
	\item A pair of elements $x,y \in \mc{B}$ have $x \preceq y$ if and only if $x \preceq_s y$ at some true stage $s$.
\end{enumerate}
The proofs of the following lemmas end up being almost exactly the same as proofs of the corresponding lemmas in the 1-decidable case, except that we talk only about true stages. We will repeat the statements of the lemmas, with the modifications to refer only to true stages.

\begin{lemma}
$(W_\rho,\preceq)$ is isomorphic to either $\omega$, $\omega^*$, or $\omega^* + \omega$. These three cases correspond, respectively, to having $\direction(s) = \RIGHT$ for all but finitely many true stages $s$, $\direction(s) = \LEFT$ for all but finitely many true stages $s$, and $\direction(s) = \RIGHT$ and $\direction(s) = \LEFT$ for infinitely many true stages $s$ each.
\end{lemma}

We say that a true stage which is also an expansionary stage is a \textit{true expansionary stage}.

\begin{lemma}\label{lem:inf-exp2}
If $\mc{A}$ is isomorphic to $\mc{D}$, then there are infinitely many true expansionary stages.
\end{lemma}

\begin{lemma}\label{lem:label2}
If there are infinitely many true expansionary stages, then every node $\rho$ or $\sigma_i^{\Box}$ and their $T$-elements have exactly the same labels. Each $\tau_{i,s}^{\Box}$ is labeled by $L$.
\end{lemma}

\begin{lemma}
Let $s$ be a true expansionary stage and suppose that $a \in \mc{A}$ and $d \in \mc{D}$ are nodes which are among the first $\extent(s)$ nodes of $\mc{A}$ and $\mc{D}$ respectively (or $T$-elements which are among the first $\extent(s)$ elements of their components, and associated to nodes which are among the first $\extent(s)$ nodes), and so that $a$ has the same labels in $\mc{A}[s-1]$ as $d$ does in $\mc{D}[s]$.
Then for any expansionary stage $s^* \geq s$, either $a$ and $d$ have the same labels in $\mc{A}[s^* - 1]$ and $\mc{D}[s]$ respectively, or one of them is labeled $L$.
\end{lemma}
%

\begin{lemma}\label{lem:bec-iso2}
Suppose that $\mc{A}$ and $\mc{D}$ are in fact isomorphic.
Let $s$ be a true expansionary stage, and let $\nu$ and $\mu$ be nodes of $\mc{A}$ and $\mc{D}$ respectively, which are among the first $\extent(s)$ nodes of those structures, and assume that neither are ever labeled $L$ at a true stage.
If, at stage $s$, $\nu$ and $\mu$ are labeled in the same way in $\mc{A}[s-1]$ and $\mc{D}[s]$ respectively, then $T_\nu \subseteq \mc{A}$ and $T_\mu \subseteq \mc{D}$ are isomorphic.
\end{lemma}

\begin{lemma}\label{lem:succ-dia2}
If $n \notin S$, then $\mc{A}$ is not isomorphic to $\mc{D}$.
\end{lemma}

\begin{lemma}\label{lem-cons-unif2}
If $n \in S$, then $\mc{A}$ has a decidable presentation which we can construct uniformly.
\end{lemma}

These last two lemmas complete the proof.

\subsection{\texorpdfstring{The $n \geq 2$ case of Theorem \ref{thm:n-pres}}{The n at least 2 case of Theorem \ref{thm:n-pres}}}\label{sec:sec-part}

Note that we built $\mc{A}$ while diagonalizing against a 2-decidable structure $\mc{D}$. So in fact we have shown that
\[ (\Sigma^1_1,\Pi_1^1) \leq_1 (I_{d-pres},\overline{I_{2-pres}}).\]
That is, for any $\Sigma^1_1$ set $S$, there is a computable function $f$ such that
\[n \in S \Longrightarrow \text{the $f(n)$th computable structure has a decidable presentation}\]
and
\[n \notin S \Longrightarrow \text{the $f(n)$th computable structure has no 2-decidable presentation}.\]
This proves the $n \geq 2$ case of Theorem \ref{thm:n-pres}.

\begin{question}\label{ques:red}
Is it true that $(\Sigma^1_1,\Pi_1^1) \leq_1 (I_{d-pres},\overline{I_{1-pres}})$?
\end{question}

\bibliography{References}

\newcommand{\etalchar}[1]{$^{#1}$}
\begin{thebibliography}{HTMMM}

\bibitem[AKMS89]{AKMS}
C.~J. Ash, J.~F. Knight, M.~Manasse, and T.~A. Slaman.
\newblock Generic copies of countable structures.
\newblock {\em Ann. Pure Appl. Logic}, 42(3):195--205, 1989.

\bibitem[And14]{Andrews}
U.~Andrews.
\newblock Decidable models of {$\omega$}-stable theories.
\newblock {\em J. Symb. Log.}, 79(1):186--192, 2014.

\bibitem[Ash86a]{AshB}
C.~J. Ash.
\newblock Recursive labelling systems and stability of recursive structures in
  hyperarithmetical degrees.
\newblock {\em Trans. Amer. Math. Soc.}, 298(2):497--514, 1986.

\bibitem[Ash86b]{AshA}
C.~J. Ash.
\newblock Stability of recursive structures in arithmetical degrees.
\newblock {\em Ann. Pure Appl. Logic}, 32(2):113--135, 1986.

\bibitem[Ash91]{Ash91}
C.~J. Ash.
\newblock A construction for recursive linear orderings.
\newblock {\em J. Symbolic Logic}, 56(2):673--683, 1991.

\bibitem[CDH08]{CholakDowneyHarrington08}
P.~A. Cholak, R.~G. Downey, and L.~A. Harrington.
\newblock On the orbits of computably enumerable sets.
\newblock {\em J. Amer. Math. Soc.}, 21(4):1105--1135, 2008.

\bibitem[CFG{\etalchar{+}}07]{CFGKKMP}
W.~Calvert, E.~Fokina, S.~S. Goncharov, J.~F.~F. Knight, O.~Kudinov, A.~S.
  Morozov, and V.~Puzarenko.
\newblock Index sets for classes of high rank structures.
\newblock {\em J. Symbolic Logic}, 72(4):1418--1432, 2007.

\bibitem[Chi90]{Chisholm}
J.~Chisholm.
\newblock Effective model theory vs.\ recursive model theory.
\newblock {\em J. Symbolic Logic}, 55(3):1168--1191, 1990.

\bibitem[CLLS00]{Cholak00}
P.~A. Cholak, S.~Lempp, M.~Lerman, and R.~A. Shore, editors.
\newblock {\em Computability theory and its applications}, volume 257 of {\em
  Contemporary Mathematics}. American Mathematical Society, Providence, RI,
  2000.
\newblock Current trends and open problems.

\bibitem[CM98]{ChisholmMoses}
J.~Chisholm and M.~Moses.
\newblock An undecidable linear order that is {$n$}-decidable for all {$n$}.
\newblock {\em Notre Dame J. Formal Logic}, 39(4):519--526, 1998.

\bibitem[DKL{\etalchar{+}}15]{DKLLMT}
R.~G. Downey, A.~M. Kach, S.~Lempp, A.~E.~M. Lewis-Pye, A.~Montalb{\'a}n, and
  D.~D. Turetsky.
\newblock The complexity of computable categoricity.
\newblock {\em Adv. Math.}, 268:423--466, 2015.

\bibitem[DM08]{DowneyMontalban08}
R.~G. Downey and A.~Montalb{\'a}n.
\newblock The isomorphism problem for torsion-free abelian groups is analytic
  complete.
\newblock {\em J. Algebra}, 320(6):2291--2300, 2008.

\bibitem[FGK{\etalchar{+}}15]{FGKKT}
E.~B. Fokina, S.~S. Goncharov, V.~Kharizanova, O.~V. Kudinov, and D.~Turetski.
\newblock Index sets of {$n$}-decidable structures that are categorical with
  respect to {$m$}-decidable representations.
\newblock {\em Algebra Logika}, 54(4):520--528, 544--545, 547--548, 2015.

\bibitem[Fok07]{Fokina07}
E.~B. Fokina.
\newblock Index sets of decidable models.
\newblock {\em Sibirsk. Mat. Zh.}, 48(5):1167--1179, 2007.

\bibitem[FS89]{FriedmanStnaley89}
H.~Friedman and L.~Stanley.
\newblock A {B}orel reducibility theory for classes of countable structures.
\newblock {\em J. Symbolic Logic}, 54(3):894--914, 1989.

\bibitem[GBM15a]{GBM}
S.~S. Goncharov, N.~A. Bazhenov, and M.~I. Marchuk.
\newblock The index set of {B}oolean algebras that are autostable relative to
  strong constructivizations.
\newblock {\em Sibirsk. Mat. Zh.}, 56(3):498--512, 2015.

\bibitem[GBM15b]{GoncharovBazhenovMarchuk15}
S.~S. Goncharov, N.~A. Bazhenov, and M.~I. Marchuk.
\newblock Index sets of constructive models of natural classes that are
  autostable with respect to strong constructivizations.
\newblock {\em Dokl. Akad. Nauk}, 464(1):12--14, 2015.

\bibitem[GN02a]{GoncharovKnight}
S.~S. Goncharov and Dzh. Na{\u\i}t.
\newblock Computable structure and antistructure theorems.
\newblock {\em Algebra Logika}, 41(6):639--681, 757, 2002.

\bibitem[GN02b]{GoncharovKnight02}
S.~S. Goncharov and Dzh. Na{\u\i}t.
\newblock Computable structure and antistructure theorems.
\newblock {\em Algebra Logika}, 41(6):639--681, 757, 2002.

\bibitem[Gon77]{Goncharov77}
S.~S. Gon{\v{c}}arov.
\newblock The number of nonautoequivalent constructivizations.
\newblock {\em Algebra i Logika}, 16(3):257--282, 377, 1977.

\bibitem[Gon80]{Goncharov80}
S.~S. Gon{\v{c}}arov.
\newblock The problem of the number of nonautoequivalent constructivizations.
\newblock {\em Algebra i Logika}, 19(6):621--639, 745, 1980.

\bibitem[Har68]{Harrison68}
J.~Harrison.
\newblock Recursive pseudo-well-orderings.
\newblock {\em Trans. Amer. Math. Soc.}, 131:526--543, 1968.

\bibitem[HKSS02]{HKSS}
D.~R. Hirschfeldt, B.~Khoussainov, R.~A. Shore, and A.~M. Slinko.
\newblock Degree spectra and computable dimensions in algebraic structures.
\newblock {\em Ann. Pure Appl. Logic}, 115(1-3):71--113, 2002.

\bibitem[HTMMM]{HTM3}
M.~Harrison-Trainor, A.~Melnikov, R.~Miller, and A.~Montalb\'an.
\newblock Computable functors and effective interpretability.
\newblock To appear in the \emph{Journal of Symbolic Logic}.

\bibitem[LS07]{LemppSlaman}
S.~Lempp and T.~A. Slaman.
\newblock The complexity of the index sets of {$\aleph_0$}-categorical theories
  and of {E}hrenfeucht theories.
\newblock In {\em Advances in logic}, volume 425 of {\em Contemp. Math.}, pages
  43--47. Amer. Math. Soc., Providence, RI, 2007.

\bibitem[Mar89]{Marker}
D.~Marker.
\newblock Non {$\Sigma_n$} axiomatizable almost strongly minimal theories.
\newblock {\em J. Symbolic Logic}, 54(3):921--927, 1989.

\bibitem[Mon]{MonICM}
A.~Montalb{\'a}n.
\newblock Computability theoretic classifications for classes of structures.
\newblock To appear in the Proccedings of the ICM 2014.

\bibitem[MPP{\etalchar{+}}]{MPPSS}
R.~Miller, J.~Park, B.~Poonen, H.~Schoutens, and A.~Shlapentokh.
\newblock A computable functor from graphs to fields.
\newblock To appear.

\bibitem[Ros82]{Rosenstein82}
J.~G. Rosenstein.
\newblock {\em Linear orderings}, volume~98 of {\em Pure and Applied
  Mathematics}.
\newblock Academic Press, Inc. [Harcourt Brace Jovanovich, Publishers], New
  York-London, 1982.

\end{thebibliography}
\bibliographystyle{alpha}

\end{document}